\documentclass[makeidx]{article}
\usepackage[latin1]{inputenc}
\title{Lower bounds for the first eigenvalue of the Laplacian with zero magnetic field in planar domains}
\author{Bruno Colbois and Alessandro Savo}
\date{}
\pdfoutput=1
\usepackage{makeidx}
\usepackage{amssymb , amsmath, amsthm,graphicx, float}
\usepackage{graphicx}
\usepackage{xcolor}
\newtheorem{defi}{Definition}
\newtheorem{theorem}[defi]{Theorem}

\newtheorem{prop}[defi]{Proposition}
\newtheorem{lemme}[defi]{Lemma}

\newcommand{\eps}{\epsilon}
 \newcommand{\twosystem}[2]{\left\{\begin{aligned} &#1\\ &#2\end{aligned}\right.}
\newcommand{\threesystem}[3]{\left\{ \begin{aligned}&#1\\ &#2\\&#3\end{aligned}\right.}

\newcommand{\nero}{\smallskip$\bullet\quad$\rm}

\newcommand{\parte}[1]{\smallskip\noindent {\rm#1)}\,\,}

\newcommand{\scal}[2]{\langle{#1},{#2}\rangle}

\newcommand{\abs}[1]{\lvert{#1}\rvert}

\newcommand{\reals}{{\bf R}}

\newcommand{\sphere}[1]{{\bf S}^{#1}}
\newcommand{\real}[1]{{\bf R}^{#1}}

\newcommand{\bd}{\partial}

\newcommand{\matrice}{\begin{pmatrix}}
\newcommand{\ok}{\end{pmatrix}}

\begin{document}

\parindent 0cm

\maketitle
\abstract{ We study the Laplacian with zero magnetic field acting on complex functions of a planar domain $\Omega$, with magnetic Neumann boundary conditions. If $\Omega$ is simply connected then the spectrum reduces to the spectrum of the usual Neumann Laplacian; therefore we focus on multiply connected domains bounded by convex curves and prove lower bounds for its ground state depending on the geometry and the topology of $\Omega$.  Besides the area, the perimeter and the diameter, the geometric invariants which play a crucial role in the estimates are the  the fluxes of the potential one-form around the inner holes and  the distance between the boundary components of the domain; more precisely, the ratio between its minimal and maximal width. Then, we give a lower bound for doubly connected domains which is sharp in terms of this ratio, and a general lower bound for domains with an arbitrary number of holes. When the inner holes shrink to  points, we obtain as a corollary a lower bound for the first eigenvalue of the so-called Aharonov-Bohm operators with an arbitrary number of poles.} 

\bigskip

\noindent {\it Classification AMS $2000$}: 58J50, 35P15\newline
{\it Keywords}: Magnetic Laplacian, spectrum, lowest eigenvalue, planar domains \newline
{\it Acknowledgments:} Research partially supported by INDAM and GNSAGA of Italy

\large
 
\noindent
\parindent 0cm

\section{Introduction}

\subsection{Definitions and state of the art}  Let $\Omega$ be a bounded, open, connected domain  with smooth boundary $\bd\Omega$ in a Riemannian manifold $(M,g)$ and let $A$ be a smooth real one-form on $\Omega$, (the {\it potential} one-form). 
Define a connection $\nabla^A$ on the space of complex-valued functions $C^{\infty}(\Omega,{\bf C})$ as follows:
$$
\nabla^A_Xu=\nabla_Xu-iA(X)u,
$$
for all vector fields $X$ on $\Omega$, where $\nabla$ is the Levi-Civita connection of $M$. The {\it magnetic Laplacian with potential $A$} is the operator acting on $C^{\infty}(\Omega,{\bf C})$:
$$
\Delta_A=(\nabla^A)^{\star}\nabla^A.
$$
In $\real n$ this gives explicitly, in the usual notation:
$$
\Delta_A=(i\nabla+A^{\sharp})^2,
$$
where $A^{\sharp}$ is the dual vector field of $A$, the {\it vector potential}.  The two-form $B=dA$ is the magnetic field; dually, in dimension $2$, $B$ is the vector field $B={\rm curl}A^{\sharp}$.

\smallskip

Scope of this paper is to discuss the spectrum of $\Delta_A$ for planar domains. Hence in what follows we take $\Omega\subset\real 2$.

\smallskip

The spectrum of the magnetic Laplacian has been studied extensively for Dirichlet boundary conditions ($u=0$ on $\bd\Omega$), and we denote by $\lambda_1^D(\Omega,A)$ the first eigenvalue. First we remark that, thanks to the diamagnetic inequality, one always has:
$$
\lambda_1^D(\Omega,A)\geq  \lambda_1^D(\Omega,0),
$$
and in particular $\lambda_1^D(\Omega,A)>0$. For planar domains and constant magnetic field (that is, $dA=B$ and $\abs{B}$ constant), a Faber-Krahn inequality holds, in the sense that the first eigenvalue of a planar domain is minimized by that of the disk of the same area (see \cite{Er1}). Estimates for sums of eigenvalues can be found in \cite{LS}.

\smallskip

However in this paper we deal with
{\it magnetic Neumann boundary conditions}, that is we impose
$
\nabla^A_Nu=0,
$
on the boundary, where $N$ is the inner unit normal to $\bd\Omega$. It is known that then $\Delta_A$ admits a discrete spectrum
$$
0\leq \lambda_1^N(\Omega,A)\leq\lambda_2^N(\Omega,A)\leq\dots
$$
diverging to $+\infty$. The first eigenvalue has the following variational characterization:

\begin{equation}
\lambda_{1}^N(\Omega,A)=\min_{u\in H_1(\Omega)\setminus\{0\}}\frac{\int_{\Omega}\abs{\nabla^Au}^2dx}{\int_{\Omega}\vert u\vert^2dx}.
\end{equation}

For computing lower bounds the diamagnetic inequality is of no use; in fact it gives:
$$
\lambda_1^N(\Omega,A)\geq\lambda_1^N(\Omega,0)=0,
$$
because $\lambda_1^N(\Omega,0)$ is simply the first eigenvalue of the usual Neumann Laplacian, which is zero (the associated eigenspace being spanned by the constant functions). 
There are fewer estimates in this regard; let us first discuss the case of a constant magnetic field $\abs{B}=B_0>0$ on planar domains. The paper \cite{EKP} gives a lower bound of $\lambda_1^N(\Omega,A)$  in terms of the inradius of $\Omega$, $\lambda_1^N(\Omega,0)$ and of course $B_0$. Asymptotic expansions as $\abs{B}\to\infty$  are obtained in 
\cite{FH3}.
We also mention the paper \cite{FH2} which investigates the validity of a reverse Faber-Krahn inequality for constant magnetic field $B_0$, that is: is it true that $\lambda_1^N(\Omega,A)$ is always bounded above by that of a disk with equal volume ? It is proved there that this inequality  is true when $B_0$ is either sufficiently small or sufficiently large, but the general case is still open in the simply connected case.

\nero In this paper we prove three lower bounds for the first eigenvalue of planar domains under Neumann conditions, when the magnetic field is identically zero. 
{\it Since this will be the only boundary condition we consider, from now on we  will simply write $\lambda_1(\Omega,A)$ instead of $\lambda_1^N(\Omega,A)$.}

\smallskip

Let us first clarify the circumstances under which the first eigenvalue might be positive even if the magnetic potential is a closed one-form on $\Omega$. This is intimately related to a phenomenon in quantum mechanics predicted in 1959 and known as {\it Aharonov-Bohm effect}, which has also experimental evidence: a particle travelling a region in the plane might be affected by the magnetic field even if this is identically zero on its path.  In fact what the particle "feels" is not the magnetic field but, rather,  the magnetic potential $A$, provided that $A$ is closed but not exact, and that the flux of $A$ around the pole  may assume non-integer values (see below for the precise condition).  

\smallskip 

Let us be more precise. From the definition we see that, if $A=0$, the spectrum of $\Delta_A$ coincides with the spectrum of the usual Laplacian under Neumann boundary conditions. The same is true when $A=df$ is an exact one-form, by the well-known {\it gauge invariance} of the magnetic Laplacian. This  fundamental  property states that the spectrum of $\Delta_{A+df}$ is the same as the spectrum of $\Delta_A$, for any $f\in C^{\infty}(\Omega)$, which follows from the identity:
$$
\Delta_Ae^{-if}=e^{-if}\Delta_{A+df}
$$
showing that $\Delta_A$ and $\Delta_{A+df}$ are unitarily equivalent. 

\smallskip

On the other hand, if the magnetic field $B=dA$ is non-zero, then $\lambda_1(\Omega,A)$ is strictly positive. One could then ask if $\lambda_1(\Omega,A)$ has to vanish whenever the magnetic field is zero, that is, whenever $A$ is a {\it closed} one-form. 

To that end, let $c$ be a closed curve in $\Omega$ (a loop). The quantity:
$$
\Phi^A_c=\dfrac{1}{2\pi}\oint_cA
$$
is called the {\it flux of $A$ across $c$} (we assume that $c$ is travelled once, and we will not specify the orientation of the loop; this will not affect any of the statements, definitions or results which we will prove in this paper). 

It turns out that 

\nero $\lambda_1(\Omega,A)=0$ if and only if $A$ is closed and the cohomology class of $A$ is an integer, that is, the flux of $A$ around any loop is an integer. 

\medskip

This was first observed by Shigekawa \cite{Sh} for closed manifolds, and then proved in \cite{HHHO} for manifolds with boundary. This remarkable feature of the magnetic Laplacian shows its deep relation with the topology of the underlying manifold $\Omega$. In this paper we will focus precisely on the situation where the potential one form is {\it closed}, and we will then give two lower bounds for the first eigenvalue $\lambda_1(\Omega,A)$. 

\smallskip

Let us then recall a few previous  results when the magnetic field is assumed to vanish.
A lower bound for a general Riemannian cylinder (i.e. the surface $\sphere 1\times (0,L)$ endowed with a Riemannian metric) and zero magnetic field has been given in \cite{CS1}, and is somewhat the inspiration of this work: one of two main results here is in fact to improve such bound  when $\Omega$ is a doubly connected planar domain.

\smallskip

Directly related to the Aharonov-Bohm effect, we mention the papers \cite{AFNN} and \cite{NNT} which investigate the behavior of the spectrum  of a domain with a pole $\Omega\setminus\{a\}$ 
when the pole $a$ approaches the boundary, for Dirichlet boundary conditions. We remark here that the pole is a distinguished point $a=(a_1,a_2)$ and the potential is the harmonic one-form:
$$
A_a=\frac 12\Big(-\dfrac{x_2-a_2}{(x_1-a_1)^2+(x_2-a_2)^2}dx_1+
\dfrac{x_1-a_1}{(x_1-a_1)^2+(x_2-a_2)^2}dx_2\Big)
$$
which has flux $\frac 12$ across any closed curve enclosing $a$, giving rise to a magnetic field which is a Dirac distribution concentrated at the pole $a$ (therefore, the magnetic field indeed vanishes on $\Omega\setminus\{a\}$).  The magnetic Laplacian $\Delta_{A_a}$ acting on $\Omega\setminus\{a\}$ is often called an  Aharonov-Bohm operator. One could think to a domain with a pole as a doubly connected domain for which the inner boundary curve shrinks to a point. 

\smallskip

 We will in fact give a lower bound for the first eigenvalue of Aharonov-Bohm operators with many poles, and Neumann boundary conditions (see Theorem \ref{punctured}).

\smallskip

 The Aharonov-Bohm operators play an interesting role in the study of minimal partitions, see chapter 8 of \cite{BH}. 
 
For Neumann boundary conditions, we mention the paper \cite{HHHO}, where the authors study the multiplicity and the nodal sets corresponding to the ground state  for non-simply connected planar domains with harmonic potential.
For doubly connected domains, it is shown that $\lambda_1(\Omega,A)$ is maximal precisely when $\Phi^A$ is congruent to $\frac 12$ modulo integers (this fact is no longer true when there are more than two holes). The proof relies on a delicate argument involving the nodal line of a first eigenfunction and the conclusion does not follow from a specific comparison argument, or from an explicit lower bound.

The focus of this paper is on lower bounds for multiply connected planar domains and zero magnetic field defined by the closed potential form $A$.  By what we have just said, it is clear that estimating the first eigenvalue  is a trivial problem if $\Omega$ is simply connected, because then any closed one-form is automatically exact, and therefore $\lambda_1(\Omega,A)=0$ by gauge invariance.  Therefore, we restrict our study to domains with $n$ holes, with $n\geq 1$. 

\smallskip

In this paper we will prove: an improved lower bound for doubly connected domains; a general lower bound for  multiply connected domains with an arbitrary number of convex holes; a lower bound for a general convex domain with an arbitrary number of punctures. Let us describe these results in detail.

\subsection{A lower bound for doubly connected domains}

Let us start from doubly connected domains ($n=1$) hence domains of type:
$$
\Omega=F\setminus\bar G,
$$
where $F$ and $G$ are open and smooth. We assume $F$ and $G$ {\it convex}. Let $\Phi^A$ be the flux of the closed potential $A$ around the inner boundary curve $\bd G$: by Shigekawa result, the lower bound is simply zero when $\Phi^A$ is an integer. Then, to hope for a positive lower bound, we need to measure how much $\Phi^A$ is far from being an integer, and the natural invariant will then be:
$$
d(\Phi^A,{\bf Z})=\min\{\abs{\Phi^A-k}: k\in{\bf Z}\}.
$$
The second important ingredient for our lower bounds is the ratio $\frac{\beta}{B}$ between the minimal width and the maximal width of $\Omega$.  To be more precise,  let us say that the line segment $\sigma\subset\Omega$ is an {\it orthogonal ray} if it hits the inner boundary $\bd G$ orthogonally. By definition, the {\it minimal width} $\beta$ (resp. {\it maximal width} $B$) of $\Omega$ is the minimal (resp. maximal ) length of an orthogonal ray contained in $\Omega$:

\begin{figure}[H]                                                                                                  
\begin{center}
 \includegraphics[width=8cm]{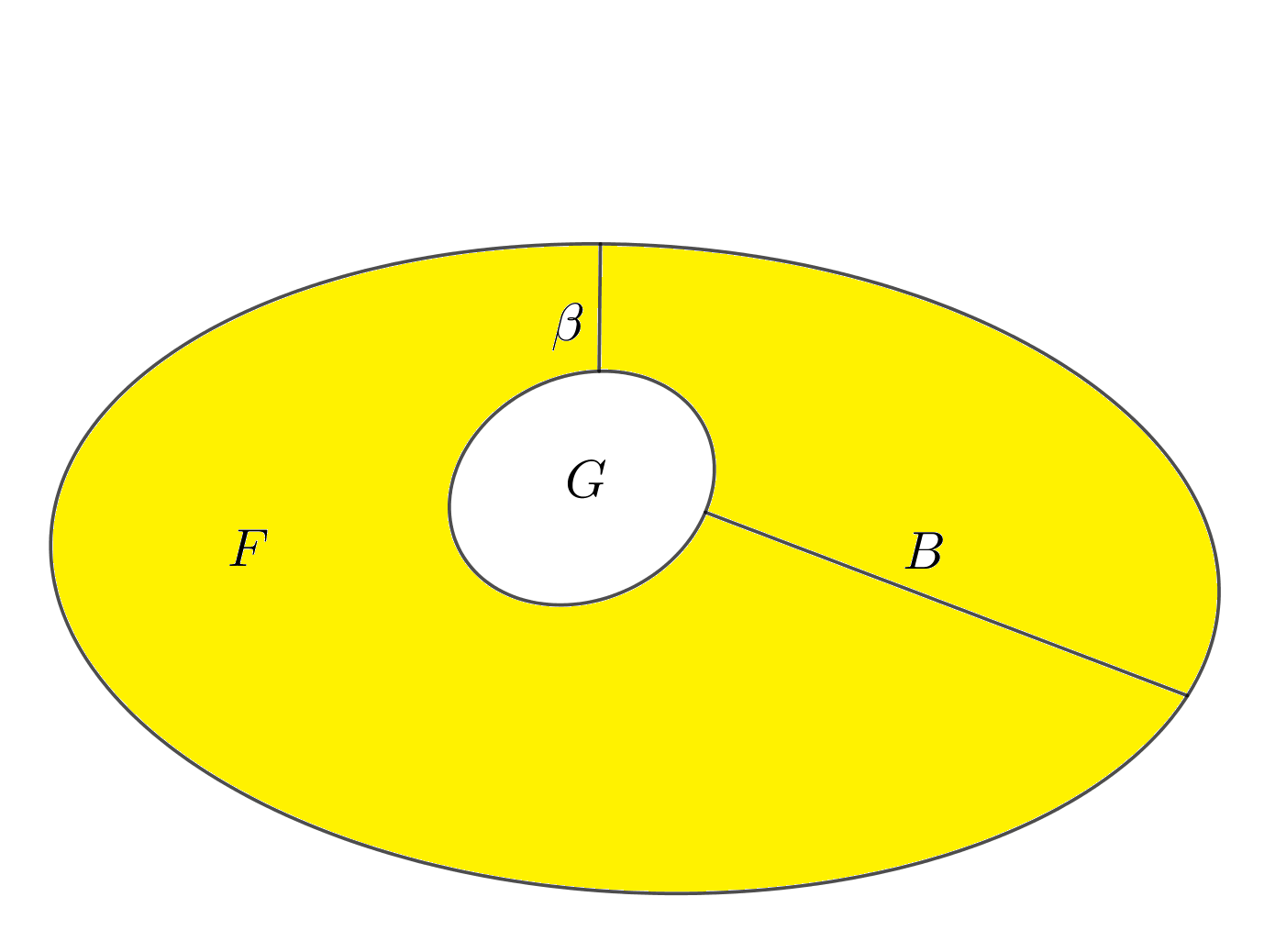}
\caption{The minimal width $\beta$ and the maximal width $B$ of an annulus $\Omega=F\setminus\bar G$}
\end{center}
\end{figure}

Note that the ratio $\frac{\beta}{B}$ is invariant under homotheties, and reaches its largest value $1$ whenever the boundary components are parallel curves.

In Theorem 2 of \cite{CS1} we prove the lower bound:
\begin{equation}\label{JFA}
\lambda_1(\Omega,A)\geq\dfrac{4\pi^2}{\abs{\bd F}^2}\dfrac{\beta(\Omega)^2}{B(\Omega)^2}d(\Phi^A,{\bf Z})^2.
\end{equation}
 We insist on the fact that if $\frac{\beta}B$ is bounded below away from zero we get a positive, uniform lower bound even if $\beta$ tends to zero. Think for example to a concentric annulus $\Omega$ of radii $1$ and $1+\beta$; then $\frac{\beta}{B}=1$ and as $\beta\to 0$ the lower bound will approach $4\pi^2d(\Phi^A,{\bf Z})^2$, a positive number, which is just the first eigenvalue of the unit circle.

\smallskip

This means that (for fixed perimeter) in order to get $\lambda_1$ small, the ratio $\frac{\beta}{B}$ (and not just $\beta$) has  to be small. 

\smallskip

{\bf Sharpness in terms of $\frac{\beta}B$.} In \cite{CS1} we showed that if $\frac{\beta}B$ is small then the first eigenvalue could indeed be small. 
We then looked for an example which could show that the dependance on $\frac{\beta^2}{B^2}$ is sharp, and we could not find it. Rather, in Examples 14 and 15 in \cite{CS1}, we constructed examples of domains such that $B$ is bounded below, say by $1$, $\abs{\bd F}$ is bounded above, $\beta$ goes to zero and $\lambda_1(\Omega,A)$ goes to zero proportionally to $\beta$, for any non-integral flux. Therefore, if one could replace $\frac{\beta^2}{B^2}$ by the linear factor $\frac{\beta}{B}$ in \eqref{JFA}, one would obtain a sharp inequality (with respect to $\frac{\beta}{B}$). See Figure \ref{sharpness} below for the example which shows sharpness. 

\smallskip

This is in fact possible, and the theorem which follows should be regarded as the first main theorem of this paper. 

\begin{theorem}\label{leading} Let $\Omega=F\setminus \bar G$ be an annulus in the plane, with $F$ and $G$ convex with piecewise-smooth boundary. Let $A$ be a closed one-form with flux $\Phi^A$ around the inner hole $G$. Then:

\parte a One has the lower bound:
$$
\lambda_1(\Omega,A)\geq \dfrac{\pi^2}{8}\cdot\dfrac{\abs{F}^2}{\abs{\bd F}^2D(F)^4}\cdot \frac{\beta(\Omega)}{B(\Omega)}\cdot d(\Phi^A,{\bf Z})^2.
$$
where $\beta(\Omega)$ and $B(\Omega)$ are, respectively, the minimum and maximum width of $\Omega$, and $D(F)$ is the diameter of $F$.

\parte b If the outer boundary $\bd F$ is smooth, and if $\beta(\Omega)$ is less than the injectivity radius of the normal exponential map of $\bd F$, then we have the simpler lower bound:
\begin{equation}\label{simple}
\lambda_1(\Omega,A)\geq \dfrac{\pi^2}{\abs{\bd F}^2}\frac{\beta(\Omega)}{B(\Omega)} d(\Phi^A,{\bf Z})^2,
\end{equation}
\end{theorem}

Note that, modulo a  factor of $4$, b) is formally identical to \eqref{JFA} with $\beta/B$ replacing $\beta^2/B^2$.

We observe that there is no positive constant $c$ such that
$$
\frac{\beta(\Omega)}{B(\Omega)}\geq c \dfrac{\abs{F}^2}{D(F)^4}
$$
for all doubly convex annuli in the plane (otherwise, the lower bound would be independent on the inner hole, and this is impossible).  This means that Theorem \ref{leading} is not a trivial consequence of \eqref{JFA}. 

In fact, the proof of Theorem \ref{leading} uses a suitable partition of $\Omega$ into overlapping annuli for which $\frac{\beta}{B}$ is, so to speak, as small as possible (see Section 2 below, and in particular Figure 3 for an example). 
Recall the $\delta$-interior ball condition:

\smallskip

{\it given $x\in\bd F$, there is  a ball of radius $\delta$ tangent to $\bd F$ at $x$ and entirely contained in $F$}.

\smallskip

Here and for further applications, we say that the injectivity radius of $\bd F$ is ${\rm Inj}(\bd F)$ if $F$ satisfies the $\delta$-interior ball condition for any $\delta\leq {\rm Inj}(\bd F)$. If $\bd F$ is smooth, its injectivity radius is positive.

\smallskip  

Finally we picture below the family of domains $\Omega_{\eps}$ which realize sharpness. $\Omega_{\eps}$ is the difference between two rectangles with parallel sides, with boundaries being $\eps$ units apart. Hence $\beta(\Omega_{\eps})=\eps$ and $B(\Omega_{\eps})$ is uniformly bounded above by $\sqrt 5$. 

\begin{figure}[H]                                                                                          
\begin{center}
 \includegraphics[width=12cm]{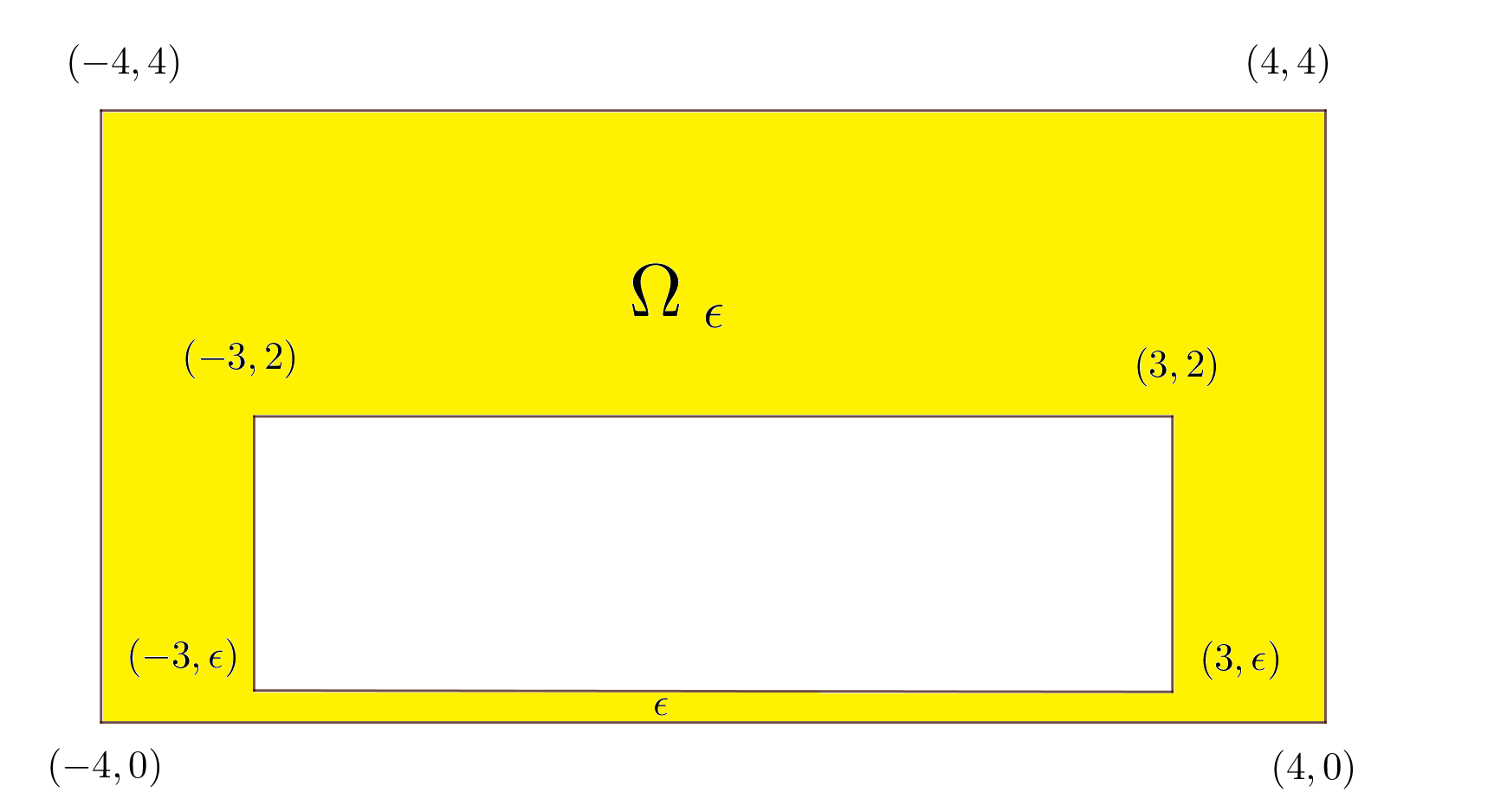}
\caption{\label{sharpness} The domain $\Omega_{\eps}$ has minimal width $\eps$ and lowest eigenvalue going to zero proportionally with $\eps$.}
\end{center}
\end{figure}

We show in Section \ref{simpleexample} that 
$$
 \frac{\pi^2}{360\sqrt 5}d(\Phi^A,{\bf Z})^2 \leq \dfrac{\lambda_1(\Omega_{\epsilon},A)}{\epsilon}\leq \frac{1}{10},
$$
so that $\lambda_1(\Omega,A)$ goes to zero proportionally to $\eps\sim\frac{\beta}{B}$.


\subsection{A general lower bound for multiply connected domains} 

Now let $\Omega$ be an $n$-holed planar domain, which we write as follows: 
\begin{equation}\label{nholed}
\Omega=F\setminus (\bar G_1\cup\dots\cup \bar G_n)
\end{equation}
where the inner holes $G_1,\dots,G_n$ are smooth, open and disjoint. We furthermore assume that  $F, G_1,\dots,G_n$ are convex. Note that:
$$
\bd\Omega=\bd F\cup\bd G_1\cup\dots\cup\bd G_n.
$$
We will call $\bd G_1\cup\dots\cup\bd G_n$ the {\it inner boundary} of $\Omega$.
The minimal and maximal widths of $\Omega$ are defined as in the case $n=1$, namely $\beta$ is the minimal length of a line segment contained in $\Omega$ and hitting the inner boundary orthogonally, and the maximal length of such line segments is by definition the maximal width $B$. 

\smallskip

It is clear that we could replace $B(\Omega)$ by the diameter of $F$, and $\beta(\Omega)$ by the invariant:
$$
\tilde\beta(\Omega)=\min\{d(\bd G_j,\bd G_k), d(\bd G_h,\bd F): j\ne k,h=1,\dots, n\}.
$$

In this section we give a lower bound of $\lambda_1(\Omega,A)$ when $\Omega$ has an arbitrary number of convex holes. 

Here is the estimate.
\begin{theorem}\label{nholes} Let $\Omega=F\setminus (\bar G_1\cup\dots\cup \bar G_n)$ be an n-holed planar domain,  where $F, G_1,\dots,G_n$ are smooth, open and convex. Let $A$ be a closed potential having  flux $\Phi_j$ around 
the $j$-th inner boundary curve $\bd G_j$, for $j=1,\dots,n$, and let
$
\gamma=\min_{j=1,\dots,n}d(\Phi_j,{\bf Z}).
$
Then we have:
\begin{equation}\label{multiply}
\lambda_1(\Omega,A)\geq \dfrac{\pi^2}{2\Big(\abs{\bd F}+2\pi B(\Omega)\Big)^2}\dfrac{\beta(\Omega)^4}{B(\Omega)^4}\cdot\gamma^2.
\end{equation}
where $\beta(\Omega)$ and $B(\Omega)$ are, respectively, the minimal and maximal width of $\Omega$.
\end{theorem}

The proof uses a suitable decomposition of $\Omega$ into a finite union of annuli, and a lower bound proved in \cite{CS1} for annuli whose outer boundary is star-shaped with respect to the inner boundary curve. A stronger estimate is proved when the inner holes are disks of the same radius (see Theorem \ref{smalldisks}).


\subsection{A lower bound for Aharonov-Bohm operators with many poles}

The power $\frac{\beta^4}{B^4}$ in the previous estimate is probably not sharp; it appears to be there for technical reasons. By shrinking the inner boundary curves to points we obtain an estimate in terms of $\frac{\beta^2}{B^2}$, which has an interesting interpretation in terms of Aharonov-Bohm operators with many poles. 

Precisely, we fix a convex domain $\Omega$ and choose $n$ points inside it, say $\mathcal P=\{p_1,\dots,p_n\}$. Consider the punctured  domain
$\Omega\setminus\mathcal P$. Given a closed one-form $A$, we define:
$$
\lambda_1(\Omega\setminus\mathcal P,A)=\liminf_{\delta\to 0}\lambda_1(\Omega\setminus\mathcal P(\delta),A)
$$
where $\mathcal P(\delta)$ is the $\delta$-neighborhood of $\mathcal P$ (it obviously consists of a finite set of disks of radius $\delta$). It is not our scope in this paper to investigate the convergence in terms of $\delta$; however, what we are looking at could be interpreted as the first eigenvalue of a Aharonov-Bohm operator with poles $p_1,\dots,p_n$ and Neumann boundary conditions. 
The proof of the theorem in the previous section simplifies, to give a general lower bound in terms of the distance between the poles, and the distance of each pole to the boundary. To that end, define:
$$
\twosystem
{\beta(\mathcal P)=\min\{d(p_j,p_k), d(p_m,\bd\Omega): p_j\ne p_k,p_m\in\mathcal P\}}
{B(\mathcal P)=\max\{d(p_j,p_k), d(p_m,\bd\Omega): p_j\ne p_k,p_m\in\mathcal P\}}
$$
Of course $B({\mathcal P})$ could be conveniently bounded above by the diameter of $\Omega$.
Let $A$ be as usual a closed one-form having flux $\Phi_j$ around the pole $p_j$. Then we have:

\begin{theorem}\label{punctured} Let $\Omega$ be a convex domain and $\mathcal P=\{p_1,\dots,p_n\}$ a finite set of poles. For the punctured domain $\Omega\setminus\mathcal P$ we have the bound:
$$
\lambda_1(\Omega\setminus\mathcal P,A)\geq\dfrac{4\pi^2}{\abs{\bd\Omega}^2}\dfrac{\beta(\mathcal P)^2}{B({\mathcal P})^2}\gamma^2.
$$
where $\gamma=\min_{j=1,\dots,n}d(\Phi_j,{\bf Z})$, and $\Phi_j$ is the flux of the closed potential $A$ around $p_j$.
\end{theorem}

\nero In a forthcoming paper, we will give upper bounds for the Laplacian with zero magnetic field on multiply connected planar domains, which are closely related to the topology (number of holes) of the domain. 

\smallskip

The rest of the paper is devoted to the proof of Theorems \ref{leading},\ref{nholes} and \ref{punctured}.

\section{Proof of Theorem \ref{leading}}\label{sectiontwo}

The proof depends on a suitable way to partition our domain $\Omega$. We first remark the simple fact that the first eigenvalue of a domain is controlled from below by the smallest first eigenvalue of the subdomains of a partition of 
$\Omega$ (Proposition \ref{simple1}). Then, we need to extend inequality \eqref{JFA} to piecewise-smooth boundaries, see Section \ref{pws}. In Section \ref{preparatory} we state our main geometric facts, Lemma  \ref{second} and Lemma \ref{secondb}, and then we prove Theorem \ref{leading}  (see Section \ref{leadingproof}). Finally, in Section \ref{decomposition}, we define the partition and we prove Lemma  \ref{second} and Lemma \ref{secondb}.

\subsection{A simple lemma}
We say that the family of open subdomains $\{\Omega_1,\dots,\Omega_n\}$ is a {\it partition} of $\Omega$, if $\bar\Omega=\bar\Omega_1\cup\dots\cup\bar\Omega_n$. Thus, the members of the partition might overlap and some of the intersections $\Omega_j\cap\Omega_j$ could have positive measure. 
If furthemore  $\Omega_j\cap\Omega_k$ is empty for all $j\ne k$ then we say that the partition is {\it disjoint}.  We observe  the following  standard fact whose proof is easy:
\begin{prop}\label{simple1} Let $\{\Omega_1,\dots,\Omega_n\}$ be a partition of the domain $\Omega$. Let $A$ be any closed potential. Then, there is an index $k=1,\dots,n$ such that
\begin{equation}\label{decofirst}
\lambda_1(\Omega,A)\geq \dfrac{1}{n}\lambda_1(\Omega_k,A).
\end{equation}
If the partition is disjoint, then:
\begin{equation}\label{decosecond}
\lambda_1(\Omega,A)\geq \min_{j=1,\dots,n}\lambda_1(\Omega_j,A).
\end{equation}
\end{prop}

\begin{proof}  We start proving \eqref{decofirst}. Let $u$ be an eigenfunction associated to $\lambda_1(\Omega,A)$. We use it as test-function for $\lambda_1(\Omega_j,A)$ and obtain, for all $j$:
\begin{equation}\label{decothird}
\lambda_1(\Omega_j,A)\int_{\Omega_j}\abs{u}^2\leq\int_{\Omega_j}\abs{\nabla^Au}^2.
\end{equation}
Now 
$$
\int_{\Omega}\abs{u}^2\leq\sum_{j=1}^n\int_{\Omega_j}\abs{u}^2\leq n\int_{\Omega_k}\abs{u}^2
$$
where the index $k$ is chosen so that $\int_{\Omega_k}\abs{u}^2$ is maximum among all $j=1,\dots,n$.
Then:
$$
\begin{aligned}
\lambda_1(\Omega_k,A)\int_{\Omega}\abs{u}^2&\leq n \lambda_1(\Omega_k,A) \int_{\Omega_k}\abs{u}^2\\
&=n\int_{\Omega_k} \abs{\nabla^Au}^2\\
&\leq n\int_{\Omega}\abs{\nabla^Au}^2\\
&=n\lambda_1(\Omega,A)\int_{\Omega}\abs{u}^2
\end{aligned}
$$
That is:
$
\lambda_1(\Omega_k,A)\leq n \lambda_1(\Omega,A),
$
which is the assertion.

\smallskip

For the proof of \eqref{decosecond}, let $\lambda_{\rm min}=\min_{j=1,\dots,n}\lambda_1(\Omega_j,A)$. From \eqref{decothird} we have, for all $j$:
$$
\int_{\Omega_j}\abs{\nabla^Au}^2\geq \lambda_1(\Omega_j,A)\int_{\Omega_j}\abs{u}^2\geq  \lambda_{\rm min}\int_{\Omega_j}\abs{u}^2
$$
We now sum over $j=1,\dots,n$ and obtain
$
\int_{\Omega}\abs{\nabla^Au}^2\geq  \lambda_{\rm min}\int_{\Omega}\abs{u}^2.
$
As $u$ is a first eigenfunction the left-hand side is precisely $\lambda_1(\Omega,A)\int_{\Omega}\abs{u}^2$, and the inequality follows. 

\end{proof}


\subsection{Convex annuli with piecewise-smooth boundary}\label{pws}
From now on  $\Omega$ will be an annulus in the plane with boundary components $\Gamma_{\rm int},\Gamma_{\rm ext}$ which we assume convex and piecewise-smooth. We will write $\Omega=F\setminus \bar G$ where $F$ and $G$ are open, convex, with piecewise smooth boundary. In that case $\Gamma_{\rm int}=\bd G$ and $\Gamma_{\rm ext}=\bd F$. 

Let $p$ be a point of $\bd G$ where  $\bd G$ is not smooth ($p$ will then be called a {\it vertex}).  The {\it normal cone} of $G$ at $p$ is the set
$$
N_G(p)=\{x\in\real 2: \scal{x}{y-p}\leq 0,\quad\text{for all}\quad y\in G\}.
$$ 
Then $N_G(p)$ is the closed exterior wedge bounded by the normal lines to the two smooth curves concurring at $p$, its boundary is the broken line depicted in the figure below.  Call $\alpha_p$ its angle at $p$.

\nero We remark  the obvious fact that $0<\alpha_p<\pi$.

\medskip

We now define the minimum and maximum width in the piecewise-smooth case. These are defined  in \eqref{betapbp} and depicted in the Figure \ref{picturecone} below.

\begin{figure}[H]
\begin{center}
 \includegraphics[width=9cm]{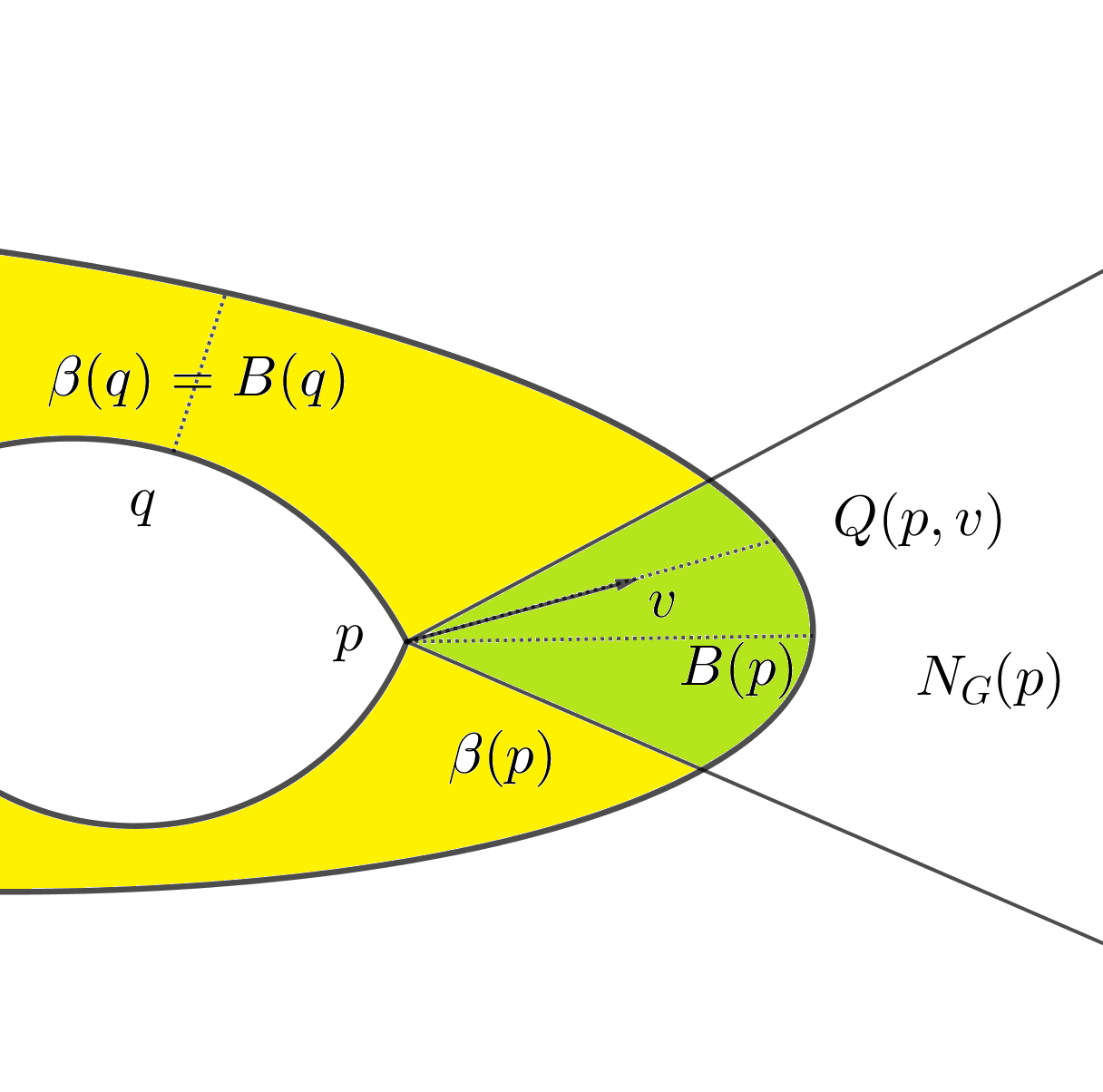}
\caption{\label{picturecone}A vertex $p$ of $\bd G$ and its normal cone $N_G(p)$.}
\end{center}
\end{figure}

 For a unit vector $v$ applied in $p$ and pointing inside $N_G(p)$ we let
$
\gamma_{p,v}(t)=p+tv
$
denote the ray exiting $p$ in the direction $v$, and let $Q(p,v)$ be the intersection 
of $\gamma_{p,v}$ with $\Gamma_{\rm ext}=\bd F$. We define:
\begin{equation}\label{betapbp}
\twosystem
{\beta(p)=\inf_{v\in N_G(p)}d(p, Q(p,v))}
{B(p)=\sup_{v\in N_G(p)}d(p, Q(p,v))}
\end{equation}
We notice that at a smooth point $q$ the cone at $q$ degenerates to the normal segment at $q$. Hence at a smooth point $q$ one has 
$$
\beta(q)=B(q).
$$

We now define
\begin{equation}\label{betab}
\twosystem
{\beta(\Omega)=\inf\{\beta(p): p\in\bd G\}}
{B(\Omega)=\sup\{\beta(p): p\in\bd G\}}
\end{equation}

$\beta(\Omega)$ and $B(\Omega)$ will be called the {\it minimum width} and, respectively, the {\it maximum width} of $\Omega$. We remark that when the two boundary components are smooth and parallel then $\beta=B$ and the ratio $\frac{\beta}{B}$ assumes its largest possible value, which is $1$.

\medskip

As a first step in the proof of Theorem \ref{leading}, we extend the inequality \eqref{JFA} to the piecewise-smooth case. 

\begin{theorem}\label{first} Let $\Omega=F\setminus \bar G$ be an annulus in the plane whose boundary components are convex and piecewise smooth. Let $\beta=\beta(\Omega)$ and $B=B(\Omega)$ be the  invariants defined in \eqref{betab}. 
Then for any closed potential having flux $\Phi^A$ around the inner boundary curve one has the lower bound:
$$
\lambda_1(\Omega,A)\geq \dfrac{4\pi^2}{\abs{\bd F}^2}\dfrac{\beta(\Omega)^2}{B(\Omega)^2} d(\Phi^A,{\bf Z})^2,
$$
where $\abs{\bd F}$ is the length of the outer boundary. 
\end{theorem}

\begin{proof}

First,  $\Omega$ admits an exhaustion by convex annuli with $C^1$-boundary, say $\{\Omega_{\epsilon}:\eps>0\}$. By that we mean:

\medskip

a) $\Omega_{\epsilon}=F_{\eps}\setminus \bar G_{\eps}$ where $F_{\eps}$ and $G_{\eps}$ are convex and have $C^1$-smooth boundary;

\medskip

b) $F_{\eps}\subseteq F$ and $G_{\eps}\supseteq G$ so that $\Omega_{\eps}\subseteq\Omega$;

\medskip

c) $\Omega=\cup_{\eps>0}\Omega_{\eps}$ and in particular $\lim_{\eps\to 0}\abs{\Omega\setminus\Omega_{\eps}}=0$.

\medskip

To construct $F_{\eps}$ we  round off corners at distance $\eps$ to each of the vertices of $\bd F$; to construct  $G_{\eps}$ we just take the convex domain bounded by the $\eps$-neighborhood of $G$.

\smallskip

Let $u$ be an eigenfunction associated to $\lambda_1(\Omega,A)$; by restriction we obtain a test-function for $\Omega_{\eps}$, hence by the min-max principle:
$$
\dfrac{\int_{\Omega_{\eps}}\abs{\nabla^Au}^2}{\int_{\Omega_{\eps}}\abs{u}^2}\geq\lambda_1(\Omega_{\eps},A).
$$
Let $\mathcal L(\Omega)$ be the functional:
$$
\mathcal L(\Omega)= \dfrac{4\pi^2}{\abs{\bd F}^2}\dfrac{\beta(\Omega)^2}{B(\Omega)^2} d(\Phi^A,{\bf Z})^2
$$
We can apply \eqref{JFA} and obtain $\lambda_1(\Omega_{\eps})\geq \mathcal L(\Omega_{\eps})$ because $\Omega_{\eps}$ has smooth boundary; then, for all $\eps>0$:
$$
\dfrac{\int_{\Omega_{\eps}}\abs{\nabla^Au}^2}{\int_{\Omega_{\eps}}\abs{u}^2}\geq 
\mathcal L(\Omega_{\eps}).
$$
We now pass to the limit as $\epsilon\to 0$ on both sides; as  $\mathcal L(\Omega_{\eps})\to \mathcal L(\Omega)$ (as we can easily see from the definitions in \eqref{betab}), we obtain the assertion:
$
\lambda_1(\Omega)\geq  \mathcal L(\Omega).
$

\end{proof}


\subsection{Preparatory results}\label{preparatory}

In this section we state the two main technical lemmas; the partition of the annulus $\Omega$ will be defined in Section \ref{decomposition}. 
\smallskip

So let $\Omega=F\setminus\bar G$ be an annulus  as above and let $\beta\in (0,\beta(\Omega)]$. We consider the distance functions:
$$
\rho_1,\rho_2: F\to [0,\infty),
$$
where $\rho_1(x)=d(x, G)$ and $\rho_2(x)=d(x,\bd F)$. Fix a parameter $\beta>0$. As $G$ is convex, with piecewise-smooth boundary, it is well-known that  the equidistants $\{\rho_1=\beta\}$ are $C^1$-smooth curves. 
We say that the parameter $\beta$ is {\it regular}  if the equidistant $\{\rho_2=\beta\}$ is a  piecewise-smooth curve. Following Appendix 2 in \cite{Sa}, we know that the set of regular parameters has full measure in $(0,\beta(\Omega)]$; as a consequence, there exists a sequence of regular parameters $\{\beta_j\}\to\beta(\Omega)$ as $j\to\infty$. 

\nero By using an obvious limiting procedure, from now on we take
$$
\beta=\beta(\Omega)
$$
and can assume that it is a regular parameter, so that $\rho_2=\beta$ is a piecewise-smooth curve. 

\begin{lemme}\label{second} Let  $\Omega=F\setminus \bar G$ be an annulus  in the plane with $F$ and $G$ convex with piecewise-smooth  boundary, and let  $\beta= \beta(\Omega)$  (which, by assumption, is a regular parameter). Then 
$\Omega$ admits a  partition $\{\Omega_1,\dots,\Omega_n\}$ into (overlapping) subdomains $\Omega_k$ with the following properties. 

\parte a $\Omega_k$ is an annulus bounded by two convex piecewise-smooth curves, that is, $\Omega=F_k\setminus \bar G_k$ with $F_k$ and $G_k$ convex, and $G_k$ contains $G$ (see figure in Section \ref{decomposition} below).

\parte b The number $n$ of annuli in the partition can be taken so that:
$$
n\leq\dfrac{2B(\Omega)}{\beta}.
$$
\end{lemme}

We estimate the ratio $\frac{\beta}{B}$ of each piece as follows.

\begin{lemme}\label{secondb}  Let $\{\Omega_1,\dots,\Omega_n\}$ be the partition in the previous lemma. 
For all $k=1,\dots,n$ one has the following facts.

 \parte a   $\abs{\bd F_k}\leq \abs{\bd F}$ and  $\beta(\Omega_k)=\beta$.
 
 \parte b  The following estimate holds:
\begin{equation}\label{inb}
\dfrac{\beta(\Omega_k)}{B(\Omega_k)}\geq\dfrac{1}{4}\dfrac{\abs{F}}{D(F)^2},
\end{equation}
where $D(F)$ is the diameter of $F$.

\parte c If $\beta(\Omega)$ is less than the injectivity radius of $\bd F$, then the following simpler lower bound holds for all $k$:
$$
\dfrac{\beta(\Omega_k)}{B(\Omega_k)}\geq \dfrac{1}{\sqrt 2}.
$$
\end{lemme}

The proof of Lemma \ref{second} and Lemma \ref{secondb} involves rather simple geometric constructions, but there are some delicate points to take care of, and will be done in Section \ref{decomposition}.  In fact, these two lemmas make it possible to write $\Omega$ as a union of subset $\Omega_k$ such that the ratio $\frac{\beta(\Omega_k)}{B(\Omega_k)}$ is bounded below, which make the proof of  Theorem \ref{leading} quite easy, as follows.


\subsection{Proof of Theorem \ref{leading}}\label{leadingproof}

We use the partition  $\{\Omega_1,\dots,\Omega_n\}$ of Lemma \ref{second}.  Let $A$ be a closed potential having flux $\Phi^A$ around the inner boundary curve $\bd G$; then, $A$ has the same flux around the inner component of $\Omega_k$, by Lemma \ref{second}a. By Proposition \ref{simple1}a there exists $k\in\{1,\dots,n\}$ such that
\begin{equation}\label{oneovern}
\lambda_1(\Omega,A)\geq \dfrac{1}{n}\lambda_1(\Omega_k,A).
\end{equation}
By Theorem \ref{first} applied to $\Omega=\Omega_k$ we see:
$$
\lambda_1(\Omega_k,A)\geq \dfrac{4\pi^2}{\abs{\bd F_k}^2}\dfrac{\beta(\Omega_k)^2}{B(\Omega_k)^2} d(\Phi^A,{\bf Z})^2.
$$
By b) of Lemma \ref{secondb} we see:
\begin{equation}\label{betab1}
\dfrac{\beta(\Omega_k)^2}{B(\Omega_k)^2}\geq\dfrac{1}{16}\dfrac{\abs{F}^2}{D(F)^4}.
\end{equation}
This, together with the inequality $\abs{\bd F_k}\leq\abs{\bd F}$, gives:
$$
\lambda_1(\Omega_k,A)\geq \dfrac{\pi^2}{4}\cdot\dfrac{\abs{F}^2}{\abs{\bd F}^2D(F)^4}\cdot d(\Phi^A,{\bf Z})^2.
$$
We insert this inequality in \eqref{oneovern} and use the inequality $\frac{1}n\geq \frac{\beta}{2B(\Omega)}$ (see Lemma \ref{second}b) to conclude that 
$$
\begin{aligned}
\lambda_1(\Omega,A)&\geq  \frac{\beta}{2B(\Omega)}\lambda_1(\Omega_k,A)\\
&\geq \dfrac{\pi^2}{8}\cdot\dfrac{\abs{F}^2}{\abs{\bd F}^2D(F)^4}\cdot\dfrac{\beta}{B(\Omega)}\cdot d(\Phi^A,{\bf Z})^2.
\end{aligned}
$$
This proves part a) of Theorem 1. 

If $\beta(\Omega)$ is less than the injectivity radius of $\bd F$ we proceed as before, using the lower bound
$
\frac{\beta(\Omega_k)}{B(\Omega_k)}\geq \frac{1}{\sqrt 2}
$
proved in Lemma \ref{secondb}c. We arrive easily at the inequality:
$$
\lambda_1(\Omega,A)\geq \dfrac{\pi^2}{\abs{\bd F}^2}\frac{\beta(\Omega)}{B(\Omega)} d(\Phi^A,{\bf Z})^2,
$$
which is Theorem 1b).


\subsection{The partition of $\Omega$ and the proof of Lemma \ref{second}}\label{decomposition}
We start by showing the partition on a particular example, see Figure \ref{pieces} below. The initial domain is a triangle $F$ minus a disk $G$ and $\beta=\beta(\Omega)$. We draw the first three pieces $\Omega_1,\Omega_2,\Omega_3$ and then the last one, which is $\Omega_6$ and which coincides with the $\beta$-neighborhood of the exterior boundary $\bd F$ (this is always the case). Note that the pieces overlap, hence the partition is not disjoint.

\begin{figure}[H] 
\begin{center}
 \includegraphics[width=15cm]{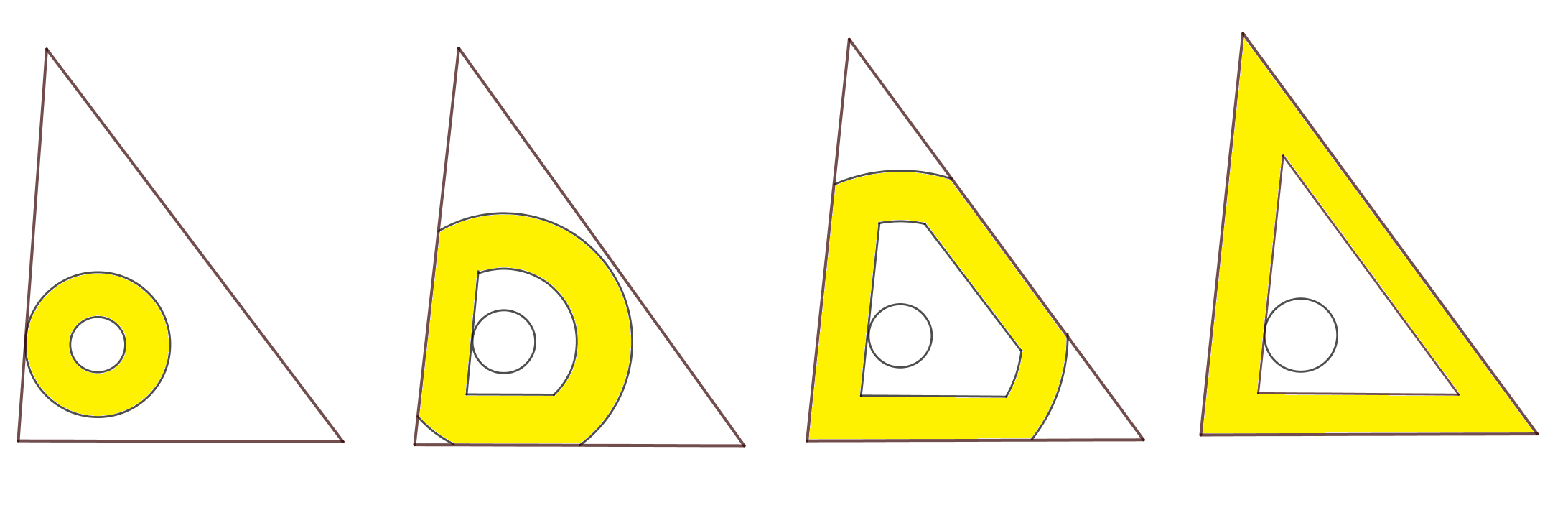}
\caption{\label{pieces}The pieces $\Omega_1, \Omega_2, \Omega_3$ and the last piece $\Omega_6=\{\rho_2<\beta\}$ when the initial domain is the triangle minus the small disk  and $\beta=\beta(\Omega)$}.
\end{center}
\end{figure}
We now proceed to construct the partition in general. Let then $\Omega$ be convex annulus $\Omega=F\setminus\bar G$ as above, and consider the distance functions:
$$
\rho_1,\rho_2: F\to [0,\infty),
$$
where $\rho_1(x)=d(x,G)$ and $\rho_2(x)=d(x,\bd F)=d(x,F^c)$.

\nero At step 1, we let $F_1=\{\rho_1<\beta\}$, $G_1=G$  and $\Omega_1=F_1\setminus \bar G_1$. That is, $\Omega_1$ is simply the subset of $F$ at distance less than $\beta$ to $G$.

\nero At step 2, we let $F_2=\{\rho_1<2\beta\}$ and 
$G_2=\{\rho_1<\beta\}\cap \{\rho_2>\beta\}$ and define $\Omega_2=F_2\setminus \bar G_2$. 

\nero  At the arbitrary step $k$, we let $F_k=\{\rho_1< k\beta\}$ and 
$G_k=\{\rho_1<(k-1)\beta\}\cap \{\rho_2>\beta\}$, and define:
$$
\Omega_k=F_k\setminus \bar G_k.
$$
Observe that for any choice of positive numbers $a,b$ the sets $\{\rho_1<a\}$ and $\{\rho_2< b\}$ are convex. Therefore,  both $F_k$ and $G_k$ are convex; moreover $G_k\subset F_k$ and then $\Omega_k$ is an annulus. This proves part a) of Lemma \ref{second}.

\smallskip

For b) we first prove the following fact:

\medskip

{\bf Fact.} {\it Let $n\doteq \Big[\dfrac{B(\Omega)}{\beta}\Big]$ (the smallest integer greater than or equal to $\frac{B(\Omega)}{\beta}$). Then $F_n=F$ and $G_n=\{\rho_2>\beta\}$. In particular, $\Omega_n=\{\rho_2<\beta\}$ and then, starting from $n$, the sequence $\Omega_n$ stabilizes: $\Omega_n=\Omega_{n+1}=\dots$.}

\smallskip

For the proof we first observe that, from the definition of $B(\Omega)$, we have $F\subseteq \{\rho_1< B(\Omega)\}$; 
then, if we fix $n\geq\frac{B(\Omega)}{\beta}$ we have by definition $F\subseteq F_n$ hence $F=F_n$.  To show that $G_n=\{\rho_2>\beta\}$ it is enough to show:
\begin{equation}\label{rhotwobeta}
\{\rho_2> \beta\}\subseteq \{\rho_1 <(n-1)\beta\}.
\end{equation}
In fact, if not,  there would be a point $x\in F$ such that  $d(x,\bd F)=\rho_2(x)>\beta$ and $\rho_1(x)\geq (n-1)\beta$. Let $y\in\bd G$ be a point at minimum distance to $x$, and prolong the segment from $y$ to $x$ till it hits $\Gamma_{\rm ext}=\bd F$ at the point $z$. It is clear that then $d(y,z)=d(y,x)+d(x,z)>n\beta$. By definition of $B(\Omega)$ we then have:
$$
B(\Omega)\geq d(y,z)>n\beta,
$$
which contradicts the definition of $n$. Hence \eqref{rhotwobeta} holds. 

\medskip

We now prove part b) of Lemma \ref{second}.  Observe that  $\bar\Omega=\cup_{k=1}^n\bar\Omega_k$ and (by the definition of $n$)
$\frac{B(\Omega)}{\beta}\geq n-1$. Since $n-1\geq \frac n2$ for all $n\geq 2$ we see that, for all $n$:
$$
\frac{B(\Omega)}{\beta}\geq \frac n2,
$$
which gives the assertion.

\subsection{Proof of Lemma \ref{secondb}}

We now study the typical piece $\Omega_k=F_k\setminus \bar G_k$  in the partition. Observe that
$$
\bd F_k=\bd_1F_k\cup\bd_2 F_k, \quad \bd G_k=\bd_1G_k\cup\bd_2 G_k,
$$ 
where
$$
\twosystem
{\bd_1F_k=\{\rho_1=k\beta\}\cap\bar F}
{\bd_2 F_k=\{\rho_1\leq k\beta\}\cap\bd F}
\quad
\twosystem
{\bd_1G_k=\{\rho_1=(k-1)\beta\}\cap\{\rho_2\geq\beta\}}
{\bd_2 G_k=\{\rho_1\leq (k-1)\beta\}\cap \{\rho_2=\beta\}}
$$
(some of these boundary pieces may be empty). As the equidistants $\{\rho_1=r\}$ are $C^1-$smooth, and the equidistant $\{\rho_2=\beta\}$ is piecewise smooth, we see that $\bd F_k$ and $\bd G_k$ are both piecewise smooth, hence 

\nero $\bd\Omega_k$ is piecewise smooth. 

\medskip

The inner boundary is $\bd G_k$, it is piecewise smooth with vertices in the set
$$
S=\{\rho_1=(k-1)\beta\}\cap\{\rho_2=\beta\}.
$$

\smallskip

Now we have to estimate the ratio $\beta(\Omega_k)/B(\Omega_k)$ for a fixed $k=1,\dots,n$. Recall the function 
$$
\beta (x):\bd G_k\to\reals
$$
defined in \eqref{betapbp}.
 First notice that the regular parts
$$
\bd_{1,{\rm reg}}F_k=\{\rho_1=k\beta\}\cap F\quad\text{and}\quad 
\bd_{1,\rm reg} G_k=\{\rho_1=(k-1)\beta\}\cap\{\rho_2>\beta\}
$$ 
are parallel, at distance $\beta$ to each other. Hence
\begin{equation}\label{betaequalsb}
\frac{\beta(x)}{B(x)}=1
\end{equation}
at all points $x\in\bd_{1,\rm reg} G_k$. Similarly, the regular sets 
$$
\bd_{2,{\rm reg}}F_k=\{\rho_1< k\beta\} \cap \bd F\quad\text{and}\quad 
\bd_{2, \rm reg} G_k=\{\rho_1<(k-1)\beta\}\cap \{\rho_2=\beta\}
$$
are parallel at distance $\beta$ and $\frac{\beta(x)}{B(x)}=1$ on $\bd_{2, \rm reg} G_k$.

Therefore,  it only remains to control the ratio $\frac{\beta(x)}{B(x)}$ at the vertices of $\bd G_k$, which are finite, say $p_1,\dots, p_m$.  

\smallskip

Each break point $p_j$ gives rise to a corresponding {\it wedge} $W(p_j)\doteq N_G(p_j)\cap\Omega$. 
In Figure \ref{typicalwedge} we enlarge the domain $\Omega_2$ relative to the partition of Figure \ref{pieces} and we show its set of wedges.
In other words, every annulus $\Omega_k$ is made up of strips of constant width $\beta$ and wedges, and we need to control 
$\frac{\beta(x)}{B(x)}$ only at the wedges.

\begin{figure}[H]
\begin{center}  
\includegraphics[width=14cm] {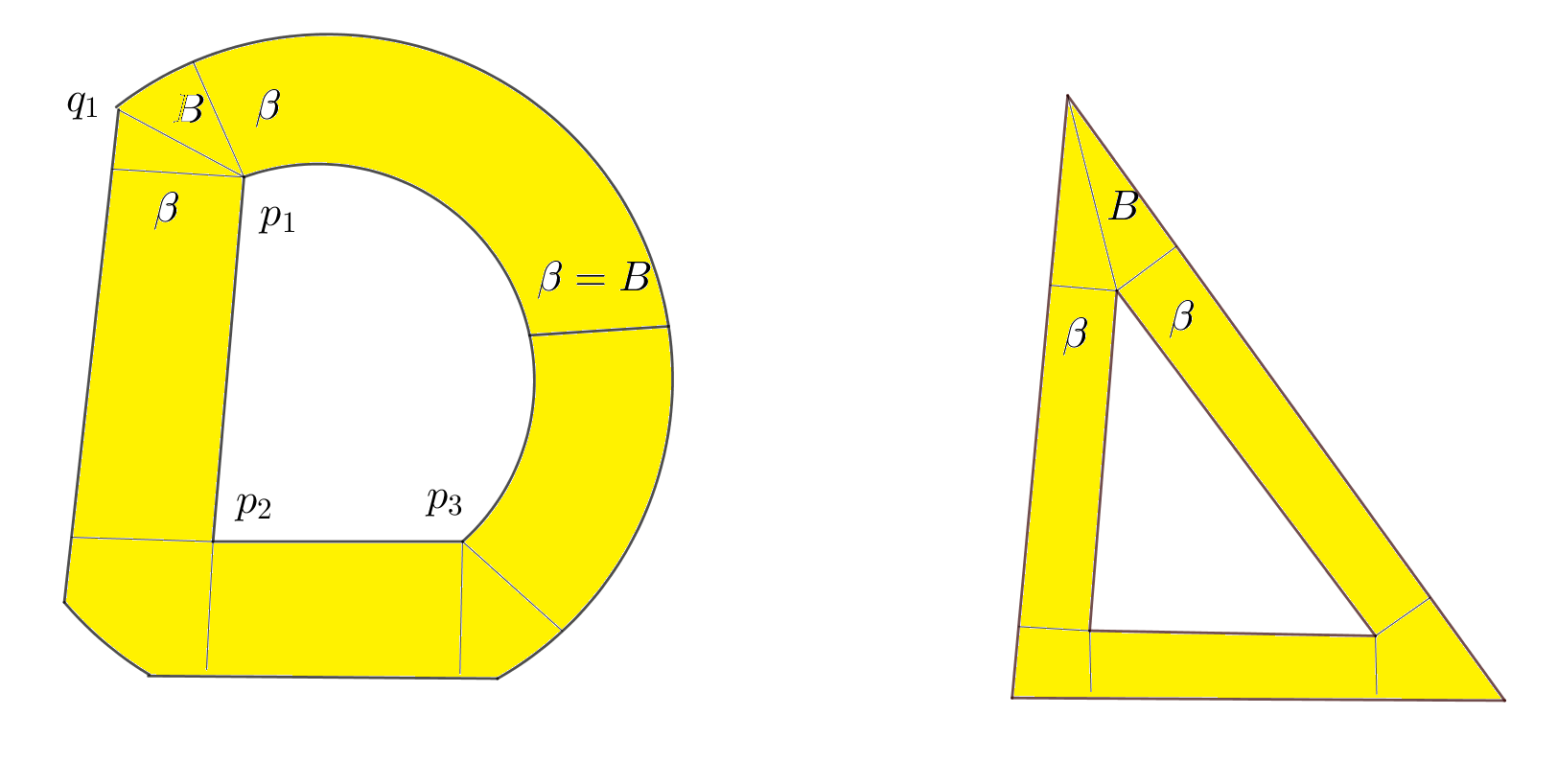}
\caption{\label{typicalwedge}On the left: the piece $\Omega_2$ and its wedges. On the right: the last piece $\Omega_6=\{\rho_2<\beta\}$; the ratio $\frac{\beta}{B}$ is small at the upper wedge, because the angle of the wedge at its break point is near $\pi$.}
\end{center}
\end{figure}

\smallskip
Typically, the ratio $\frac{\beta}{B}$ is small when there are small angles; nevertheless, this ratio is controlled from below by the diameter and the volume of the outer domain, as we will see in the next section.

\smallskip

As $F_k$ is a convex subset of $F$, we see that $\abs{\bd F_k}\leq\abs{\bd F}$. Now it is clear from the construction that $\beta(x)\geq \beta$ for all $x\in\bd F_k$; moreover, the inequality is attained. Therefore
$$
\beta(\Omega_k)=\beta
$$
for all $k$. This proves part a) of Lemma \ref{secondb}.


\subsection{End of proof of Lemma \ref{secondb}} The estimate $\frac{\beta}{B}$ on the wedges of the generic piece $\Omega_k$ will be a consequence of Lemma \ref{types} below.  

We recall that the cut-locus of $\bd F$ is the closure of the set of all points which can be joined to $\bd F$ by at least two minimizing segments;  moreover, the injectivity radius of $\bd F$ is the minimum distance of $\bd F$ to the cut-locus. If $\bd F$ is smooth, its injectivity radius is positive. Finally the distance function $d(\cdot,\bd F)$ is smooth outside the cut-locus.

Then, we fix a piece $\Omega_k$ and recall that $\beta(\Omega_k)=\beta$. 
For simplicity we write $B(\Omega_k)=B$ and recall that, by definition, we have $B\geq\beta$.

\smallskip

Let $\{p_1,\dots,p_n\}$ be the vertices of $\bd G_k$. For $p$ in this set, write $p=\gamma_1\cap\gamma_2$, where $\gamma_1,\gamma_2$ are the arcs concurring at $p$. Note that either $\gamma_j$ is an equidistant to $\bd G$, that is, is a subset of $\rho_1=(k-1)\beta$, and in that case we say that $\gamma_j$ is {\it parallel to $\bd G$}, or $\gamma_j$ is an equidistant to $\bd F$, that is, is a subset of $\rho_2=\beta$, and in that case we say that $\gamma_j$ is {\it parallel to $\bd F$}. There are two possibilities:

\smallskip

{\bf Type 1.} The vertex $p=\gamma_1\cap\gamma_2$ where $\gamma_1$ is parallel to $\bd G$ and $\gamma_2$ is parallel to $\bd F$;

\smallskip

{\bf Type 2.} $\gamma_1$ and $\gamma_2$ are both parallel to $\bd F$.

\smallskip

Note that the second type corresponds to the situation where the vertex $p$ is a point of the cut-locus of $\bd F$. The situation where $\gamma_1$ and $\gamma_2$ are both parallel to $\bd G$ does not occur, because then $p$ would belong to the cut locus of $\bd G$; however the cut-locus of a convex domain is always contained in the interior of the domain; as $p$ is outside $G$  this is impossible. 

\nero For the partition in the example  and its piece $\Omega_2$ (see Figure 5), the vertices $p_1$ and $p_3$ are of type 1, while the vertex $p_2$ is of type 2. 

\begin{lemme}\label{types}

\parte a If $p$ is of type 1 then the interior angle of $G_k$ at $p$ is larger than or equal to $\frac{\pi}2$, hence the angle of the wedge $W(p)$ at $p$ is at most $\frac {\pi}2$. Consequently,
$$
\frac{\beta}{B(p)}\geq \frac{1}{\sqrt 2}.
$$

\parte b If $p$ is of type 2, then $p$ is in the cut-locus of $\bd F$ and one has:
$$
\frac{\beta}{B(p)}\geq\dfrac{1}{4}\dfrac{\abs{F}}{D(F)^2}.
$$

\parte c If $\beta=\beta(\Omega)$ is less than the injectivity radius of $\bd F$ then the estimate in a) will hold at all vertices of the decomposition.

\end{lemme}

Since the lower bound in b) is always weaker than that in a), we have  b) at all vertices of $\bd G_k$. It is clear that Lemma \ref{types} completes  the proof of Lemma \ref{secondb}.

\smallskip

{\bf Proof of Lemma \ref{types}a)} If $p$ is of type 1 then $p$ is not on the cut-locus of $\bd F$, hence $\nabla\rho_2$ exists and is a well-defined unit vector in a neighborhood of $p$. Note that $\nabla\rho_2(p)$ points in the direction where the distance to $\bd F$ increases (obviously an analogous observation holds for $\nabla\rho_1$).   Now observe that the angle of the wedge $W(p)$ is the angle between the vectors $\nabla\rho_1$ and $-\nabla\rho_2$ (see Figure \ref{picture8a}).

\begin{figure}[H]
\begin{center}
\includegraphics[width=7cm]{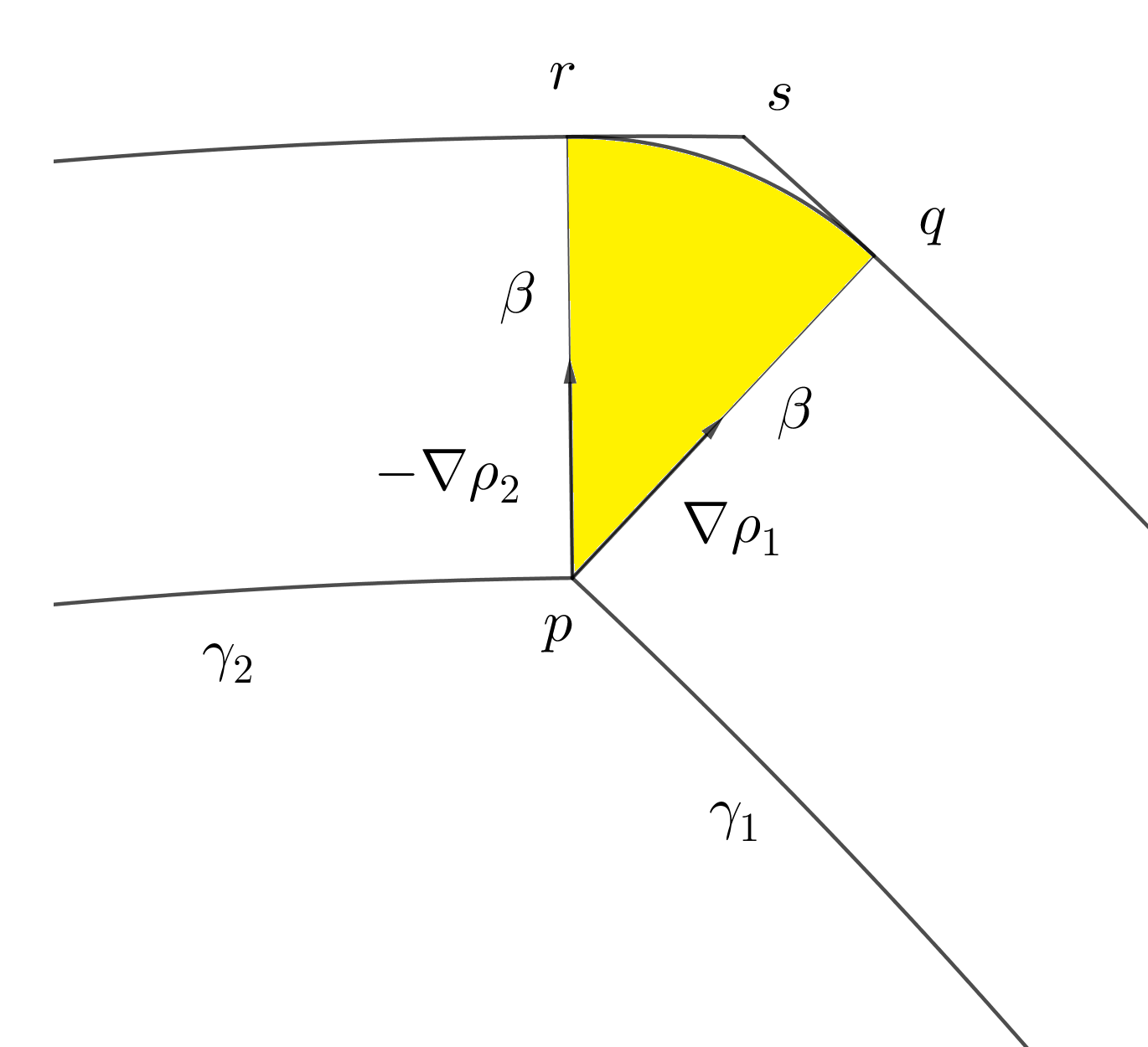}
\caption{\label{picture8a}Proof of Lemma \ref{types}a}
\end{center}
\end{figure}

Hence, it is enough to show that the quantity
$$
\scal{\nabla\rho_1(p)}{\nabla\rho_2(p)}=c(p)
$$
is non-positive. 
Assume on the contrary that $c(p)>0$. We let $\alpha(t)$ denote the segment which minimizes distance from $p$ to $\bd G$ (parametrized by arc-length); hence $\alpha'(t)=-\nabla\rho_1(\alpha (t))$. We let $f(t)$ be the function which measures distance from $\alpha(t)$ to $\bd F$, so that:
$$
f(t)= \rho_2(\alpha(t)).
$$
Now $f'(t)=\scal{\nabla\rho_2(\alpha(t))}{\alpha'(t)}=-\scal{\nabla\rho_2(\alpha(t))}{\nabla\rho_1(\alpha(t)}$. In particular, 
$$
f'(0)=-c(p)<0.
$$
As $f(0)=\beta$, this means that  for small $t$ one has $\rho_2(\alpha(t))<\beta$, but this impossible because $\alpha(t)\in G_k$, and all points of $G_k$ are, by definition, at distance at least $\beta$ to $\bd F$. 

\smallskip

Hence $c_p\leq 0$ and the angle of the wedge at $p$ is at most $\frac{\pi}{2}$. Now,  the wedge $W(p)$ is contained in the polygon with vertices $p,q,s,r$ as in the picture, hence $B(p)\leq d(p,s)\leq \sqrt 2\beta$ because the angle at $p$ is at most $\frac{\pi}2$, the angles at $r$ and $s$ are $\frac{\pi}2$, and $d(p,r)=d(p,q)=\beta$.

\smallskip

{\bf Proof of b)}. Let $p$ be a vertex of type 2: then, the two arcs concurring at $p$ are parallel to $\bd F$, and  $p$ belongs to the cut locus of $\bd F$. The boundary of the wedge $W(p)$ is made of two distinct segment of the same length $\beta$ minimizing distance to $\bd F$. Then, $W(p)$ is contained in a wedge of the last member of the partition, that is, $\{\rho_2<\beta\}$. As $B(p)$ depends only on $W(p)$, we could as well estimate the ratio $\frac{\beta}{B(p)}$ by estimating the corresponding ratio for the wedges of $\{\rho_2<\beta\}$, which will allow to express $\frac{\beta}{B(p)}$ in terms of the geometry of $\{\rho_2<\beta\}$, hence, in terms of the geometry of $F$.

The relevant picture is shown below (see Figure \ref{picturetypical}),
in which we evidence such an edge $W$ (dark shadowed in the picture):  it has its vertex in the point $p$ of $\rho_2=\beta$;  we let $q\in W$ be a point such that $d(p,q)=B$. We omit to draw the inner boundary as it will play no role in the proof. 

\smallskip

Let $T$ be the triangle with dotted boundary, with a vertex in $q$ and such that 
$F\setminus W\subseteq T$.  As $\phi+\gamma$ is the exterior angle at a vertex of the piecewise-smooth curve $\rho_2=\beta$, we see that $\phi+\gamma\leq \pi$.

\begin{figure}[H]
\begin{center}
 \includegraphics[width=8cm]{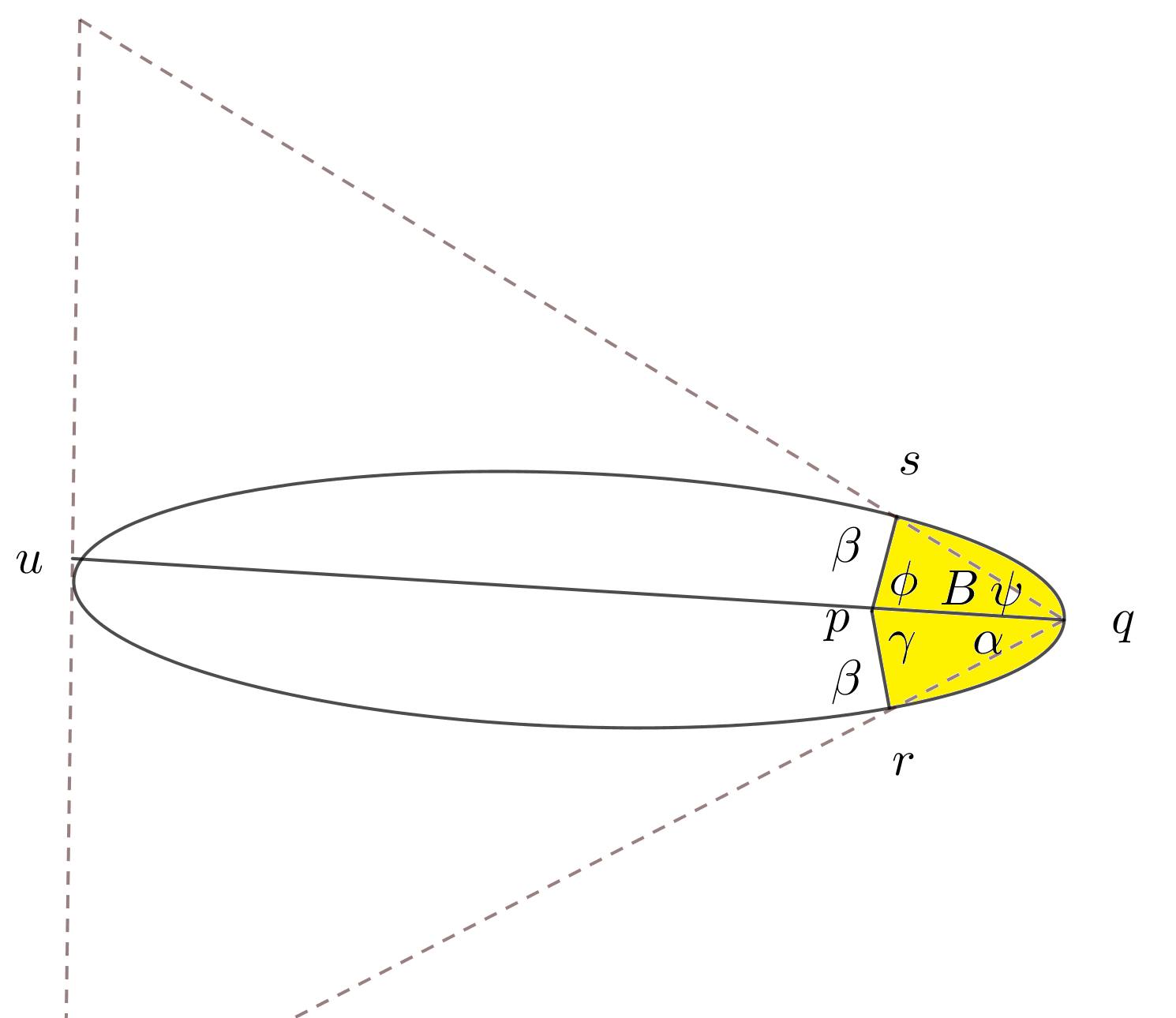}
  \includegraphics[width=6cm]{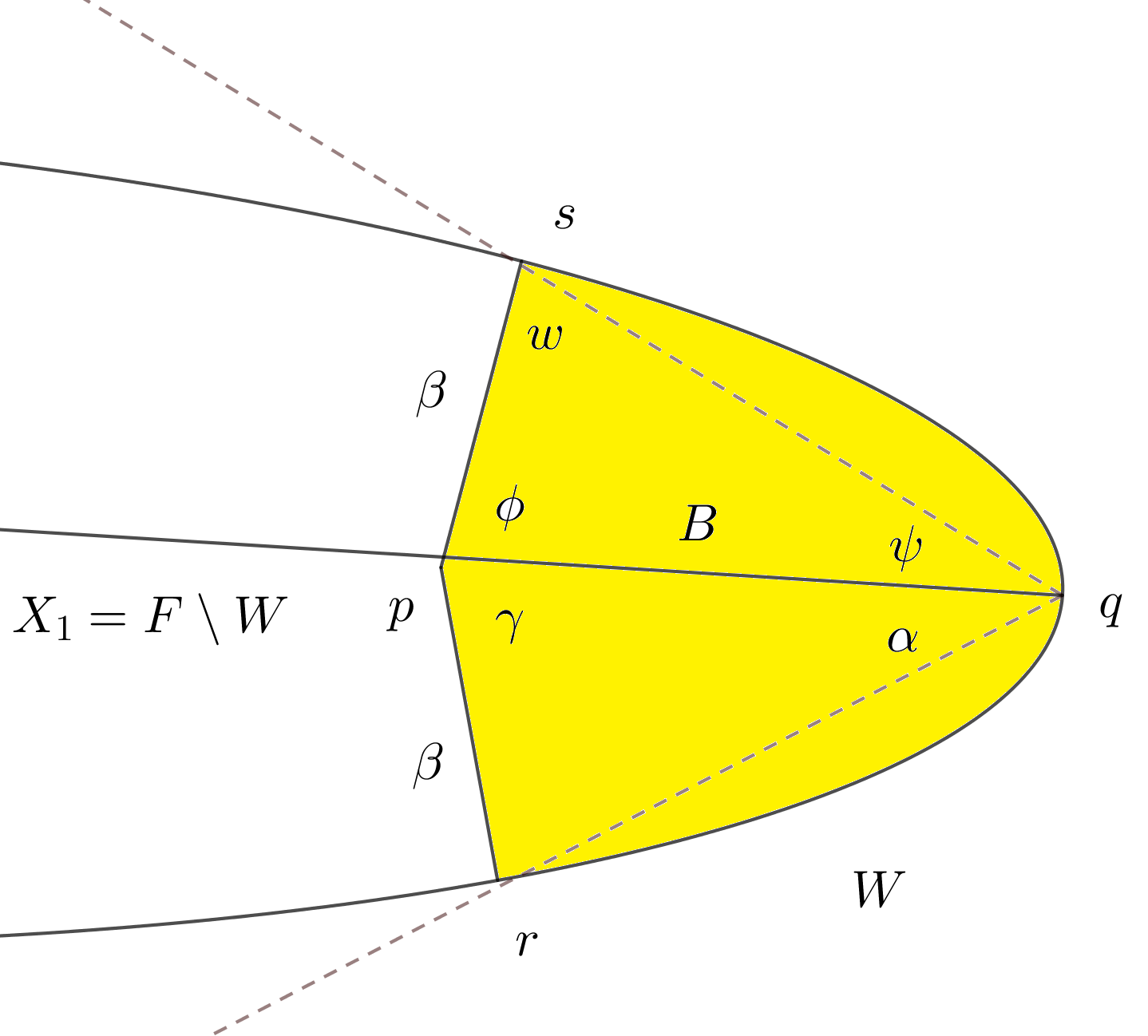}
\caption{\label{picturetypical}Estimate of $\beta/B$ on a typical wedge  }
\end{center}
\end{figure}


{\it Each of the angles $\psi$ and $\alpha$ is less than $\frac{\pi}{2}$}.  
Consider the circle with center $p$ and radius $\beta$. If $q$ is inside this circle then $B<\beta$ which is impossible. Then $q$ is outside the circle; $\alpha$ and $\psi$ are, each, less then the corresponding angles at the vertex $q'$ obtained by projecting $q$ on the circle. It is clear that each of these two angles is less than $\pi/2$.

\smallskip

Let $w$ be the angle at the vertex $s$. Then 
\begin{equation}\label{sinpsi}
\dfrac{\beta}{\sin\psi}=\dfrac{B}{\sin w}\geq B,
\end{equation}
and similarly
$
\frac{\beta}{\sin\alpha}\geq B.
$
If $\psi\geq\frac{\pi}4$ or $\alpha\geq\frac{\pi}4$ then $\frac{\beta}{B}\geq\frac{1}{\sqrt 2}$ and we are finished because $\frac{1}{\sqrt 2}>\frac{\abs{F}}{4D(F)^2}$.

\nero Hence we can assume from now on $\alpha\leq \psi\leq\frac{\pi}4$.

\begin{lemme}\label{gbone} In the above notation we have
$
\abs{T}\geq \dfrac{\sqrt 2}{4}\abs{F}.
$
\end{lemme}

\begin{proof} We first remark that $\frac{\pi}{4}\leq\phi\leq\frac{3\pi}{4}$. In fact,  we have $\psi\leq\frac{\pi}4$ and $w\leq \frac{\pi}2$ so that 
$
\phi\geq\frac{\pi}4.
$
On the other hand, the same argument applies to $\gamma$, that is, $\gamma\geq\frac{\pi}4$. Therefore, as $\phi+\gamma\leq \pi$ we conclude
$
\phi\leq\frac{3\pi}4.
$
The same bounds are satisfied by $\gamma$.

\medskip

As $W$ is contained in the union of two parallelograms with sides $\beta$ and $B$ we see:
\begin{equation}\label{bbeta}
\abs{W}\leq 2B\beta.
\end{equation}
We set $X_1=F\setminus W$; we also let $P_2$  be the convex polygon with vertices $p, q, r, s$ and $P_1=T\setminus P_2$ (see Figure \ref{polygons}).

\begin{figure}[H]
\begin{center}
\includegraphics[width=6cm]{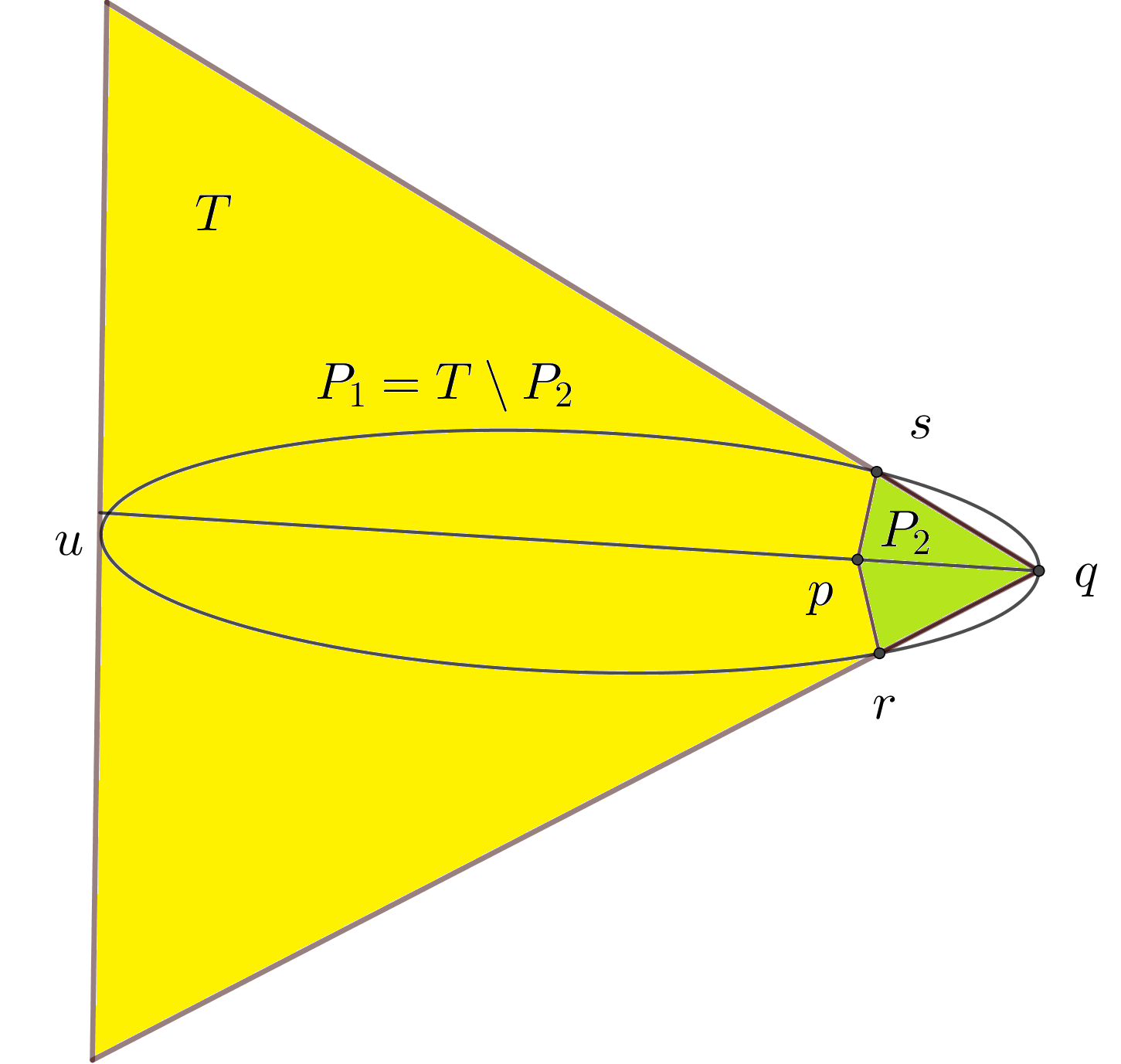}
\caption{\label{polygons}The polygons $P_1$ and $P_2$.}
\end{center}
\end{figure}

We have disjoint unions:
$$
\twosystem
{F=X_1\cup W}
{T=P_1\cup P_2}
$$
We will show that
$$
\twosystem
{\abs{P_1}\geq \abs{X_1}\geq \dfrac{\sqrt 2}{4}\abs{X_1}}
{\abs{P_2}\geq \dfrac{\sqrt 2}{4}\abs{W}}
$$
and the assertion will follow by summing up the two inequalities. 
Now the first inequality is obvious, because $X_1\subseteq P_1$. 
By the bounds for $\phi$ and $\gamma$ we see that 
$\sin\phi$ and $\sin\gamma$ are both, at least, $\frac{1}{\sqrt 2}$. Then:
$$
\begin{aligned}
\abs{P_2}&= \dfrac 12 B\beta\sin\phi+\dfrac 12 B\beta\sin\gamma\\
&\geq \dfrac{\sqrt 2}{2}B\beta
\end{aligned}
$$
Combining the two estimates we see that
$
\abs{P_2}\geq \dfrac{\sqrt 2}{4}\abs{W}
$
as asserted.
\end{proof}


\medskip

{\bf End of proof of Lemma \ref{secondb}.} Refer to Figure \ref{polygons}. We can assume that $\alpha\leq\psi\leq\frac{\pi}{4}$. We let $\delta$ be the length of the segment joining $q$ and $u$ (which meets the side opposite to $q$ orthogonally, by definition), so that:
$$
\abs{T}=\frac 12\delta^2(\tan\alpha+\tan\psi)\leq \delta^2\tan\psi.
$$
The assumptions give $\tan\psi\leq\sqrt 2\sin\psi$, so that
$
\abs{T}\leq \sqrt 2\delta^2\sin\psi.
$
Using the lower bound for $\abs T\geq \frac{\sqrt 2}{4}\abs{F}$ proved before, we have
$
\sin\psi\geq \dfrac{\abs{F}}{4\delta^2}
$
and then, from \eqref{sinpsi}:
$$
\frac{\beta}{B}\geq\sin\psi\geq \dfrac{\abs{F}}{4\delta^2}\geq \dfrac{\abs{F}}{4D(F)^2},
$$
the last inequality holding because evidently $\delta\leq D(F)$. This proves Lemma \ref{types}b and, with it, Lemma \ref{secondb} is completely proved.


\subsection{Example showing sharpness} \label{simpleexample}This example is taken from \cite{CS1}, we repeat it below for the sake of clarity. Its scope is to show that the inequality of Theorem \ref{leading} is sharp in $\frac{\beta}B$.

\smallskip

We take $F$ to be the rectangle $[-4,4]\times [0,4]$, $G_{\epsilon}= [-3,3]\times [\epsilon,2]$ and consider the doubly connected domain:
$$
\Omega_{\epsilon}=F\setminus \bar G_{\epsilon}.
$$
We refer to the picture in the Introduction. 
We let $A$ be any closed $1$-form. As a direct consequence of the gauge invariance of the magnetic Laplacian, it is proved in \cite{CS1} that, for any planar domain $\Omega$ one has:
\begin{equation}\label{excision}
\lambda_1(\Omega,A)\leq \nu_1(D),
\end{equation}
where $D$ is any closed, simply connected subdomain of $\bar\Omega$, and where $\nu_1(D)$ denotes the first eigenvalue of the usual Laplacian with Neumann boundary conditions on $\bd D\cap\bd\Omega$ and with Dirichlet boundary conditions on $\bd D\cap\Omega$. 

\medskip

Given our choice of $\Omega_{\epsilon}$, we remove from it the rectangle 
$(-1,1)\times (0,\epsilon)$ to get the simply connected subdomain called $D_{\epsilon}$. 
We estimate $\nu_1(D_{\epsilon})$ by taking the test-function as follows:
$$
\phi(x,y)=\threesystem
{1\quad\text{on the complement of $[-2,-1]\times[0,\epsilon]\cup[1,2]\times[0,\epsilon]$ }}
{x-1\quad\text{on $[1,2]\times[0,\epsilon]$}}
{1-x\quad\text{on $[-2,-1]\times[0,\epsilon]$}}
$$
It is readily seen that $\int_{D_{\epsilon}}\abs{\nabla\phi}^2=2\epsilon$, while 
$\int_{D_{\epsilon}}\phi^2\geq C>0$ (note that $C>20$). Therefore:
$$
\nu_1(D_{\epsilon})\leq \dfrac{\epsilon}{10}.
$$
Given \eqref{excision} we conclude that:
\begin{equation}\label{upperex}
\lambda_1(\Omega_{\epsilon},A)\leq \dfrac{\beta(\Omega_{\epsilon})}{10}=\dfrac{\eps}{10}
\end{equation}

\smallskip

Applying  our lower bound in Theorem \ref{leading} to this case, we have
$$
\beta=\epsilon;\ B=\sqrt 5;\ \vert F\vert =32;\ \vert \partial F\vert = 24;\ D(F)=4\sqrt 5,
$$
and we obtain
\begin{equation}\label{lowerex}
\lambda_1(\Omega_{\epsilon},A)\ge \frac{\pi^2}{360\sqrt 5}d(\Phi^A,{\bf Z})^2 \,\epsilon.
\end{equation}
We now observe that the minimum width $\beta(\Omega_{\epsilon})=\epsilon$, by construction, and that $B(\Omega_{\eps})$ is bounded above by $4$. Taking into account \eqref{upperex} and \eqref{lowerex} we see that 
$\lambda_1(\Omega_{\eps})$ goes to zero proportionally to $\eps\sim \frac{\beta(\Omega_{\eps})}{B(\Omega_{\eps})}$.

\section{Proof of Theorem \ref{nholes}}\label{sectiontwobis}

Let $\Omega$ be an $n$-holed domain, which we write:
$
\Omega=F\setminus (\bar G_1\cup\dots\cup\bar G_n),
$
with $F,G_1,\dots,G_n$ smooth, open and convex. 

From now on we denote $\Gamma_j=\bd G_j$. The idea is to use a suitable partition of $\Omega$ by annuli  $\Omega_j$ whose boundary is either a piece of $\bd F$ or is an equidistant curve from two interior boundary  curves; each $\Omega_j$ is an annulus 
$\Omega_j=F_j\setminus G_j$ with piecewise-smooth exterior boundary $\bd F_j$ which is star-shaped with respect to $\Gamma_j=\bd G_j$. We can then apply a theorem in \cite{CS1} and obtain the uniform bound, valid for all $j$:
\begin{equation}\label{lambdaj}
\lambda_1(\Omega_j,A)\geq \dfrac{2\pi^2}{9\Big(\abs{\bd F}+2\pi B(\Omega)\Big)^2}\dfrac{\beta(\Omega)^4}{B(\Omega)^4}\cdot\gamma^2.
\end{equation}
As the bound holds for all subdomains of a disjoint partition it holds a fortiori for $\Omega$, thanks to Proposition \ref{simple1}.

\subsection{The partition of $\Omega$}

We start by giving  in Figure \ref{expartition} below the picture of the partition $\{\Omega_1,\Omega_2,\Omega_3\}$ when $\Omega$ has three holes. The inner boundary of each piece $\Omega_j$ is made of equidistant sets from two suitable holes.

\begin{figure}[H]
\begin{center}
 \includegraphics[width=13cm]{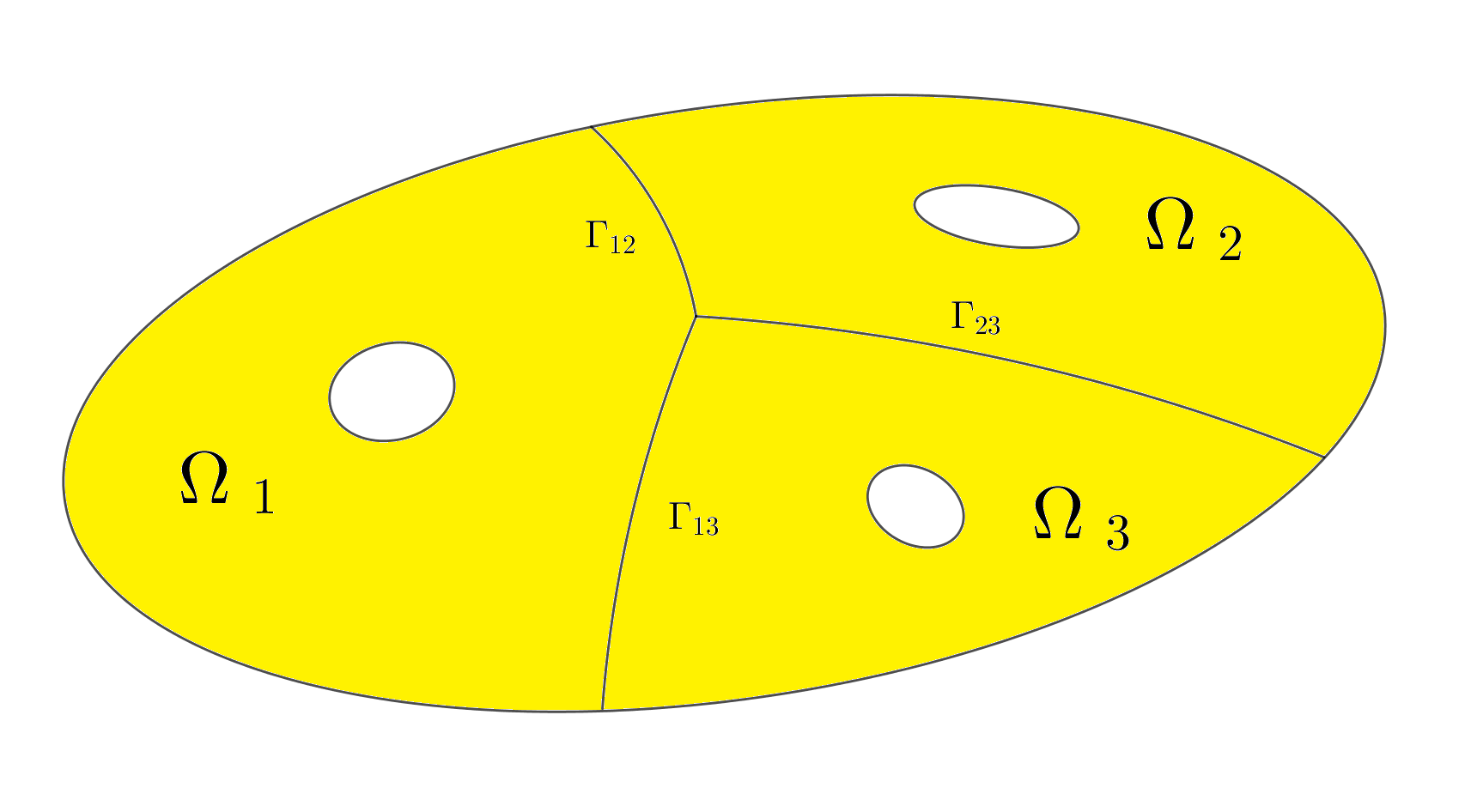}
\caption{\label{expartition}The partition $\{\Omega_1,\Omega_2,\Omega_3\}$ of a  domain $\Omega$ with three holes. The curves $\Gamma_{jk}$ are equidistant sets.}
\end{center}
\end{figure}

Here is the construction. For each $j=1,\dots,n$ we consider the non-empty open set:
$$
F_j=\{x\in F: d(x,G_j)<d(x,G_k)\quad\text{for all}\quad k\ne j\}.
$$
If we set:
\begin{equation}\label{hjk}
H_{jk}=\{x\in\real 2: d(x,G_j)<d(x,G_k)\}
\end{equation}
we see that we can write
$$
F_j=\cap_{k\ne j}(H_{jk}\cap F).
$$
It is clear that $\bar F=\cup_{j=1}^n \bar F_j$. We remark that $\bd H_{jk}$ is the equidistant set from $G_j$ and $G_k$:
\begin{equation}\label{equidistant}
\bd H_{jk}=\{x\in\real 2: d(x,G_j)=d(x,G_k)\}.
\end{equation}

We have the following general fact.

\begin{lemme} \label{decomposition1} Let $G_1$ and $G_2$ be disjoint smooth convex domains. Then the equidistant set $\bd H_{12}$ as above is a smooth curve.
\end{lemme}

\begin{proof} Let $A=\real 2\setminus (G_1\cup G_2)$ and let $\rho_j$  be the distance function to $G_j$, $j=1,2$. The convexity of $G_j$ implies that $\rho_j$ is smooth on the complement of $G_j$, so that $\rho_j$ is smooth on $A$.  Let $f=\rho_1-\rho_2$, so that $\bd H_{12}$ is the zero set of $f$. One has $\nabla f=\nabla\rho_1-\nabla\rho_2$ and it is enough to show that  $\nabla f$ has no critical points on $\{f=0\}$. Let $p$ be a point in $\bd H_{12}=\bd H_{21}$ and let  $\gamma_1$ be the line segment which minimizes the distance from $p$ to $G_1$. One has: $\gamma_1(t)=p-t\nabla\rho_1(p)$. The corresponding minimizing segment from $p$ to $G_2$ is then $\gamma_2(t)=p-t\nabla\rho_2(p)$. If $\nabla\rho_1(p)=\nabla\rho_2(p)$ then $\gamma_1(t)=\gamma_2(t)$ and, as $d(p,G_1)=d(p,G_2)$, the two minimizing segments would have the same foot $q$, which would then belong to both $G_1$ and $G_2$: but this impossible because $G_1$ and $G_2$ are disjoint.

Hence, on $\{f=0\}$ one has $\nabla f\ne 0$ which proves smoothness. 
\end{proof}

As $G_1,\dots,G_k$ are disjoint we see that $G_j\subset F_j$ and then we can introduce the annulus 
$$
\Omega_j\doteq F_j\setminus \bar G_j,
$$
that is:
$$
\Omega_j=\{x\in \Omega: d(x,G_j)<d(x, G_k)\quad\text{for all}\quad k\ne j\}.
$$
The family $\{\Omega_1,\dots,\Omega_n\}$ gives rise to a disjoint partition of $\Omega$, as the next lemma shows. 

\begin{lemme}\label{partition} The following properties hold:

\parte a $\bar\Omega=\cup_{j=1}^n\bar\Omega_j$.

\parte b For $j\ne k$ one has that $\Omega_j\cap\Omega_k=\emptyset$ and $\bar\Omega_j\cap\bar\Omega_k$ is a smooth curve (eventually empty).

\parte c $\Omega_j$ is an annulus with smooth inner boundary $G_j$ and piecewise smooth outer boundary $\bd F_j$. Moreover:
$$
\bd F_j=\Big(\cup_{k\ne j}\Gamma_{jk}\Big)\cup (\bd F\cap\bar F_j),
$$
where
$
\Gamma_{jk}=\bar F_j\cap \bar F_k=\bar\Omega_j\cap\bar\Omega_k
$
is contained in the equidistant curve $\bd H_{jk}$ from $G_j$ and $G_k$.
\end{lemme}

Note that actually $\bar\Omega_j\cap\bar\Omega_k=\bar F_j\cap \bar F_k$. The proof of the lemma is clear from the definitions. 

\subsection{Estimate of $\lambda_1(\Omega_j,A)$}

As the partition of Lemma \ref{partition} is disjoint, from Proposition \ref{simple1} we have:
$$
\lambda_1(\Omega,A)=\min_{j=1,\dots,n}\lambda_1(\Omega_j,A).
$$
Therefore, in this section, we estimate the first eigenvalue of the generic member of the partition. 
To that end, recall a relevant theorem from \cite{CS1}. Let $\Omega_1=F_1\setminus G_1$ be an annulus with inner boundary curve $\Gamma_1=\bd G_1$, where $G_1$ is smooth and convex. For $x\in\Gamma_1$ and $t\geq 0$, consider the segment $\gamma(t)=x+tN(x)$ where $N(x)$ is the unit normal to $\Gamma_1$ oriented outside $G_1$. Let $Q(x)$ be the first intersection of $\gamma_x(t)$ with the outer boundary curve $\bd F_1$, and let $\theta_x$ be the angle between $\gamma'_x$ and the outer normal $\nu$ to $F_1$ at $Q(x)$. We set:
$$
m(\Omega_1)\doteq\min_{x\in\Gamma_1}\cos\theta_x.
$$
We recall that $\Omega_1$ is said to be {\it starlike with respect to $\Gamma_1$} 
if, for any $y\in F_1$, the segment minimizing distance from $y$ to $\Gamma_1$ is entirely contained in $\Omega_1$. 

\smallskip

We also set:
$$
\twosystem{\beta(\Omega_1)\doteq\min\{d(x,Q(x): x\in\Gamma_1\}}
{B(\Omega_1)\doteq\max\{d(x,Q(x): x\in\Gamma_1\}}
$$
which are called, respectively, the minimum and maximum width of $\Omega_1$.

The estimate in Theorem 2 of \cite{CS1} says that:
\begin{equation}\label{cs}
\lambda_1(\Omega_1,A)\geq\dfrac{4\pi^2}{\abs{\bd F_1}^2}\dfrac{\beta(\Omega_1) m(\Omega_1)}{B(\Omega_1)}d(\Phi^A,{\bf Z})^2.
\end{equation}
We will apply \eqref{cs} to each annulus $\Omega_j$  in the above partition of $\Omega$. We start from:

\begin{lemme}\label{mj} Let $\Omega=F\setminus (G_1\cup\dots\cup G_n)$ and let $\Omega_j=F_j\setminus G_j$ be a piece in the partition defined above. Then:

a) $\Omega_j$ is an annulus which is starlike with respect to $\bd G_j$, and moreover:
$$
m(\Omega_j)\geq\dfrac{\beta(\Omega_j)}{2B(\Omega_j)}
$$

b) One has the estimate:
$$
\abs{\bd F_j}\leq \dfrac{2B(\Omega_j)}{\beta(\Omega_j)}\Big(\abs{\bd G_j}+2\pi B(\Omega_j)\Big)
$$
\end{lemme}

Lemma \ref{mj} allows to prove Theorem \ref{nholes} as follows. We apply \eqref{cs} to $\Omega_j$ and get:
$$
\lambda_1(\Omega_j,A)\geq \dfrac{\pi^2}{2\Big(\abs{\bd G_j}+2\pi B(\Omega)\Big)^2}\dfrac{\beta(\Omega_j)^4}{B(\Omega_j)^4}\cdot\gamma^2.
$$
To make the lower bound independent on $j$, it is enough to observe that
$\beta(\Omega_j)\geq\beta(\Omega), B(\Omega_j)\leq B(\Omega)$ and $\abs{\bd G_j}\leq \abs{\bd F}$. Then we get:
$$
\lambda_1(\Omega_j,A)\geq \dfrac{\pi^2}{2\Big(\abs{\bd F}+2\pi B(\Omega)\Big)^2}\dfrac{\beta(\Omega)^4}{B(\Omega)^4}\cdot\gamma^2.
$$
which is the final step of the proof. 

\medskip

Then, it remains to prove Lemma \ref{mj}.


\subsection{Proof of Lemma \ref{mj}a}

It is enough to prove it for $j=1$.  We first prove that $\Omega_1=F_1\setminus \bar G_1$ is star shaped with respect to $\Gamma_1=\bd G_1$. Let $y\in\bd F_1$ and let $\sigma$ be the segment starting at $y$ and minimizing distance to $\Gamma_1$: let $x\in\Gamma_1$ be the foot of $\sigma$.

\nero Note that, as $y\in\bd F_1$, we must have $d(y,\Gamma_1)\leq d(y,\Gamma_k)$ for all $k\ne 1$.  

\smallskip

We prove that $\sigma$ is entirely contained in $\Omega_1$. Assume by contradiction that there is $z\in\sigma$ such that $z\notin\Omega_1$. Then $z\in\Omega_h$ for some $h\ne 1$, and there exists $q\in\Gamma_h$ with $d(z,q)<d(z,x)$. But then:
$$
\begin{aligned}
d(y,x)&=d(y,z)+d(z,x)\\
&>d(y,z)+d(z,q)\\
&\geq d(y,q)
\end{aligned}
$$
that is, $d(y,x)>d(y,q)$ and this means that $d(y,\Gamma_1)>d(y,\Gamma_h)$,  which contradicts the assumption. Hence $\Omega_1$ is star shaped. 

\smallskip

We now estimate $\cos\theta_x$, and for convenience we refer to the picture below, Figure \ref{cutlocus}.

\begin{figure}[H]
\begin{center}
 \includegraphics[width=13cm]{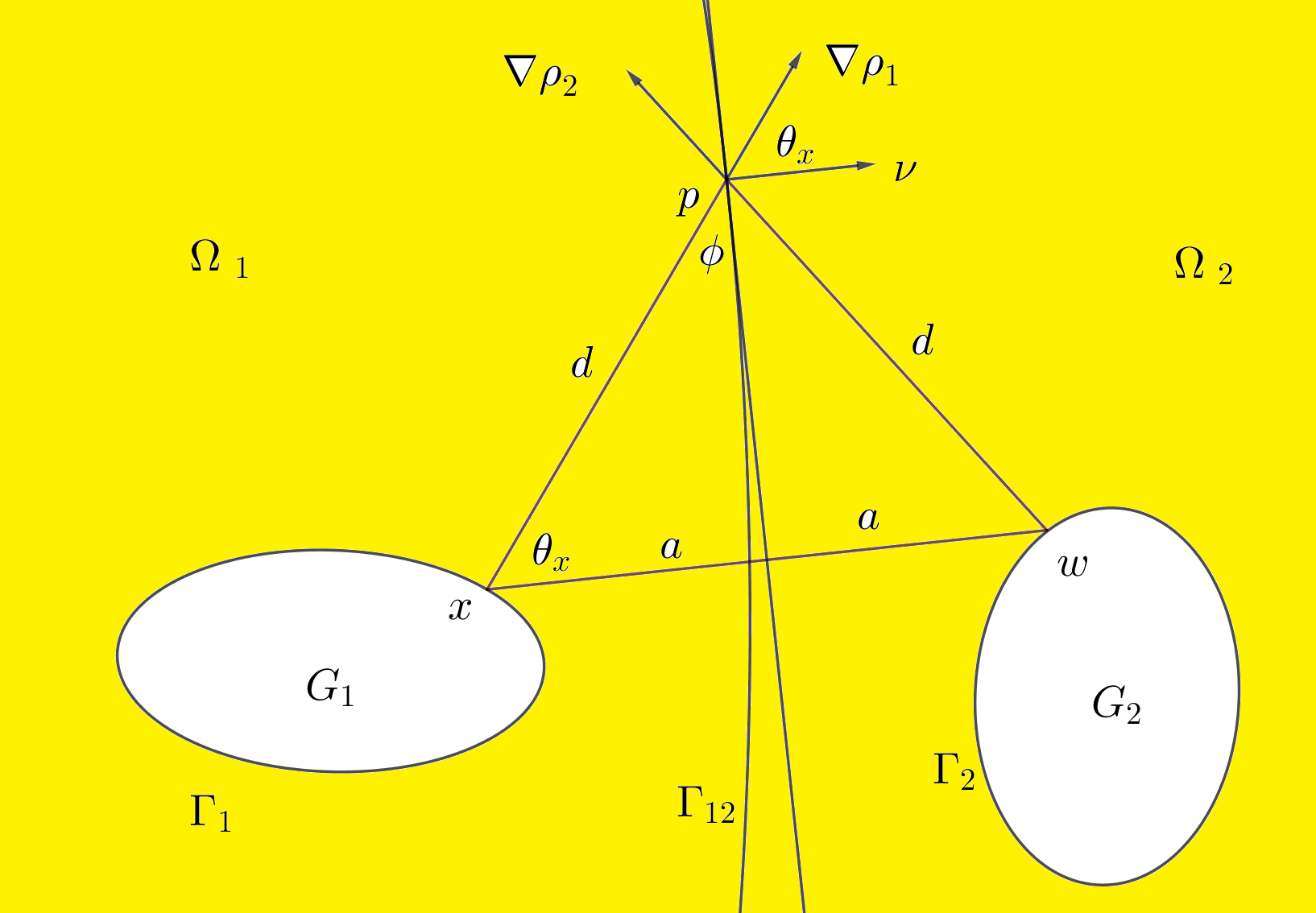}
\caption{\label{cutlocus}The curve  is the equidistant $\Gamma_{12}$ from $\Gamma_1$ and $\Gamma_2$. The tangent to $\Gamma_{12}$ at $p=Q(x)$ is the line through $p$ orthogonal to $\nu$. It cuts the angle between $\nabla\rho_1$ and $\nabla\rho_2$ in half.}
\end{center}
\end{figure}

Let $x\in\Gamma_1$ and draw the segment $\gamma_x(t)=x+tN_x$ where $N_x$ is the unit normal vector to $\Gamma_1$ pointing outside $G_1$. It hits $\bd F_1$ at the point $p=Q(x)$. If $p\in \bd F$ we proceed as in \cite{CS1} (because $\bd F$ is convex) and get
$$
\cos\theta_x\geq\dfrac{\beta(\Omega_1)}{B(\Omega_1)}.
$$
If $p\notin \bd F$ (as in the picture) then $p\in\Gamma_{1k}$ for some $k\ne 1$; we can assume that $k=2$.  Let $w$ be the point in $\Gamma_2$ such that 
$d(p,x)=d(p,w)=d$; observe that $\cos\theta_x=\scal{\nabla\rho_1}{\nu}$ where $\nu$ is the  normal to $\Gamma_{12}$ at $p$ pointing away from $\Gamma_1$. Observe that $\nu$ is the unit vector in the direction of $\nabla\rho_1-\nabla\rho_2$, and that $\nabla\rho_1+\nabla\rho_2$ is tangent to $\Gamma_{12}$ at $p$. If $2\phi$ is the angle between $\nabla\rho_1$ and $\nabla\rho_2$ then we see that $2\phi+2\theta_x=\pi$, that is
$$
\theta_x=\frac{\pi}2-\phi.
$$
Consider the triangle with vertices $x,w,p$; it is isosceles on the basis $xw$, (whose length is denoted $2a$); its height is part of the tangent line to the equidistant at $p$.  One sees that
$$
d\cos\theta_x=d\sin\phi=a
$$
hence
$$
\cos\theta_x=\frac{a}{d}.
$$
Now $2a=d(x,w)\geq \beta(\Omega_1)$ by definition of $\beta(\Omega_1)$; as the segment joining $x$ and $p$ is entirely contained in $\Omega_1$ we see that $d\leq B(\Omega_1)$. Hence
$$
\cos\theta_x\geq \frac{\beta(\Omega_1)}{2B(\Omega_1)}
$$
as asserted. 


\subsection{Proof of Lemma \ref{mj}b}

Recall that the typical piece of the decomposition is $\Omega_j=F_j\setminus G_j$. 
We need to estimate $\abs{\bd F_j}$; this is a bit more difficult now because $F_j$ is no longer convex (there are circumstances under which each $F_j$ is actually convex - for example, when all holes are disks of the same radius - and we will discuss this case in the next section, to obtain a simpler final estimate). 

Set $j=1$ for concreteness. 
We apply Green formula to the function $\rho_1(x)=d(x,\bd G_1)$. Note that $\Delta\rho_1(x)$ is the curvature at $x$ of the equidistant to $\bd G_1$ through $x$; as $\bd G_1$ is convex one has $\Delta\rho_1\leq 0$ on the complement of $G_1$.  By Green formula:
$$
\int_{\Omega_1}\Delta\rho_1=\int_{\bd\Omega_1}\scal{\nabla\rho_1}{N},
$$
where $N$ is the inner unit normal. We let
$D(G_1,B)$ denote the $B$-neighborhood of $G_1$, so that $F_1\subseteq D(G_1,B)$ by the definition of $B$. Since $\Delta\rho_1\leq 0$:
$$
\int_{\Omega_1}\Delta\rho_1\geq\int_{D(G_1,B)}\Delta\rho_1.
$$
By co-area formula:
$$
\int_{D(G_1,B)}\Delta\rho_1=\int_0^B\int_{\rho_1=r}\Delta\rho_1\,dr=-2\pi B
$$
because $\int_{\rho_1=r}\Delta\rho_1=-2\pi$ for all $r$ (we are integrating the opposite of the curvature of a closed curve, and we always obtain $-2\pi$). Therefore:
\begin{equation}\label{deltarho}
\int_{\Omega_1}\Delta\rho_1\geq -2\pi B.
\end{equation}

On the other hand $\bd\Omega_1=\bd G_1\cup\bd F_1$. Hence:
$$
\int_{\bd\Omega_1}\scal{\nabla\rho_1}{N}=\int_{\bd G_1}\scal{\nabla\rho_1}{N}+
\int_{\bd F_1}\scal{\nabla\rho_1}{N}.
$$
The first piece is $\abs{\bd G_1}$. On the outer boundary $\bd F_1$ we see that:
$$
\scal{\nabla\rho_1}{N}=-\cos\theta_x\leq -\dfrac{\beta}{2B},
$$
where $\theta_x$ is as in the proof of part a), and the inequality then follows from part a). Then:
$$
\int_{\Omega_1}\scal{\nabla\rho_1}{N}\leq \abs{\bd G_1}-\frac{\beta}{2B}\abs{\bd F_1},
$$
and given \eqref{deltarho} we obtain $-2\pi B\leq \abs{\bd G_1}-\frac{\beta}{2B}\abs{\bd F_1}$, that is:
$$
\abs{\bd F_1}\leq\frac{2B}{\beta}(\abs{\bd G_1}+2\pi B),
$$
which gives the assertion. 


\section{Proof of Theorem \ref{punctured}} About the partition $\{\Omega_1,\dots,\Omega_n\}$ of the previous section for domains with $n$ holes, we remark that if the inner holes $G_1,\dots,G_n$ are disks of the same radius $r$, then the equidistant set between any pair of them is simply a straight line, and therefore each $\Gamma_{jk}=\bar\Omega_j\cap\bar\Omega_k$ is a line segment; moreover the subdomains $F_1,\dots,F_n$ are all convex: see Figure \ref{puncture} which illustrates the partition when all holes shrink to a point.

\begin{figure}[H]
\begin{center}
 \includegraphics[width=8cm]{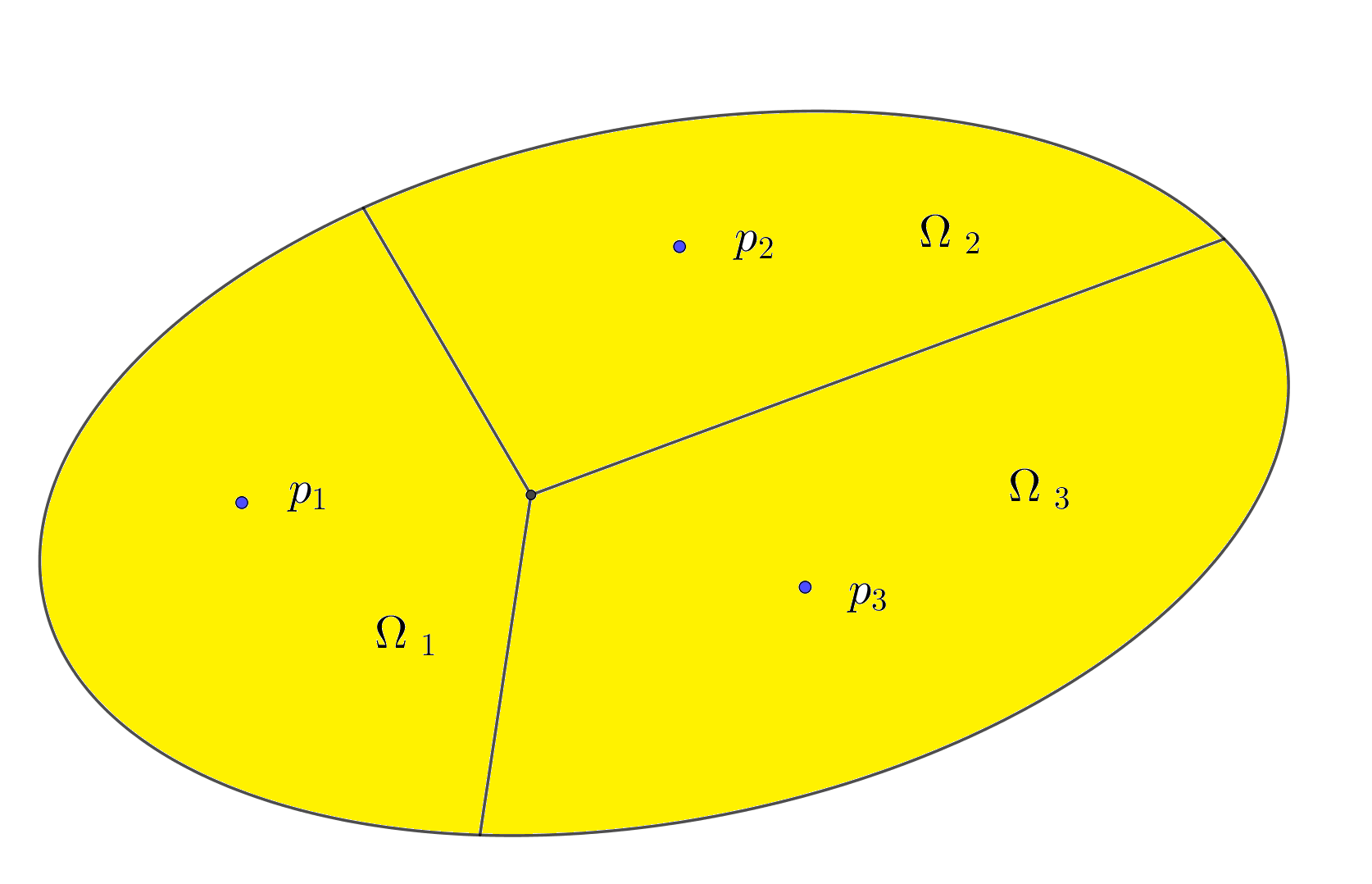}
\caption{\label{puncture}The partition $\{\Omega_1,\Omega_2,\Omega_3\}$ for a domain $\Omega$ punctured at the points $p_1,p_2,p_3$.}
\end{center}
\end{figure}

We can directly apply \eqref{JFA} to each $\Omega_j$ and obtain:
$$
\lambda_1(\Omega_j,A)\geq \dfrac{4\pi^2}{\abs{\bd F_j}^2}\dfrac{\beta(\Omega_j)^2}{B(\Omega_j)^2}\cdot d(\Phi_j,{\bf Z})^2
$$
As $F_j$ is convex, we have $\abs{\bd F_j}\subseteq\abs{\bd F}$ and therefore we arrive at the following estimate.

\begin{theorem}\label{smalldisks} Let $\Omega=F\setminus (G_1\cup\dots\cup G_n)$ with $F$ convex and $G_1, \dots, G_n$ being disjoint disks of center, respectively, $p_1,\dots,p_n$ and common radius $r>0$.  Then:
$$
\lambda_1(\Omega,A)\geq \dfrac{4\pi^2}{\abs{\bd F}^2}\dfrac{\beta(\Omega)^2}{B(\Omega)^2}\cdot \gamma^2
$$
with $\gamma=\min_{j=1,\dots,n}d(\Phi_j,{\bf Z})$. 
\end{theorem}

We remark that if we let $r\to 0$ in Theorem \ref{smalldisks} we get the lower bound for the punctured domain $\Omega\setminus\{p_1,\dots,p_n\}$ as in Theorem \ref{punctured}.

\addcontentsline{toc}{chapter}{Bibliography}
\bibliographystyle{plain}
\bibliography{biblioCSdomains1}

\bigskip

\normalsize 
\noindent Bruno Colbois \\
Universit\'e de Neuch\^atel, Institut de Math\'ematiques \\
Rue Emile Argand 11\\
 CH-2000, Neuch\^atel, Suisse

\noindent bruno.colbois@unine.ch

\medskip

\normalsize
\noindent
Alessandro Savo \\
 Dipartimento SBAI, Sezione di Matematica,
Sapienza Universit\`a di Roma\\
Via Antonio Scarpa 16\\
00161 Roma, Italy

\noindent alessandro.savo@uniroma1.it

\end{document}